\documentclass[10pt,reqno]{amsart}
\usepackage{amsmath,amsopn,amssymb,amsthm}
\usepackage[cp1251]{inputenc}
\usepackage[english]{babel}
\usepackage{graphicx}%
\usepackage{color}
\usepackage{url}

\newtheorem{theorem}{Theorem}[section]

\newtheorem{fact}[theorem]{Fact}

\newtheorem{conjecture}[theorem]{Conjecture}
\newtheorem{corollary}[theorem]{Corollary}

\newtheorem{example}[theorem]{Example}

\newtheorem{lemma}[theorem]{Lemma}

\newtheorem{remark}[theorem]{Remark}

\renewenvironment{proof}[1][Proof]{\noindent\textbf{#1.} }{\ \rule{0.5em}{0.5em}}

\begin{document}
\title[Laugwitz and Unicorn Conjectures  for
 norms with symmetries]{Proof
of Laugwitz Conjecture and Landsberg Unicorn Conjecture for Minkowski norms with   $SO(k)\times SO(n-k)$-symmetry.
}
\author{Ming Xu and Vladimir S. Matveev}

\address{Ming Xu \newline
School of Mathematical Sciences,
Capital Normal University,
Beijing 100048,
P. R. China}
\email{mgmgmgxu@163.com}

\address{Vladimir S. Matveev, Corresponding author\newline
Institut f\"{u}r Mathematik,
Fakult\"{a}t f\"{u}r Mathematik und Informatik,
Friedrich-Schiller-Universit\"{a}t Jena,
Germany
}
\email{vladimir.matveev@uni-jena.de}

\date{}

\begin{abstract} For a smooth strongly convex
Minkowski norm $F:\mathbb{R}^n \to \mathbb{R}_{\geq0}$, we  study isometries of the Hessian metric  corresponding  to  the function $E=\tfrac12F^2$. Under the additional assumption that $F$ is invariant  with respect to the standard action  of $SO(k)\times SO(n-k)$, we prove a conjecture  of  Laugwitz stated in  1965. Further, we describe all isometries between such Hessian metrics, and
prove Landsberg Unicorn Conjecture for Finsler manifolds of dimension $n\ge 3$ such that
 at every point the corresponding Minkowski norm has a linear $SO(k)\times SO(n-k)$-symmetry.


\end{abstract}\maketitle
\section{Introduction}

\subsection{Definitions and state of the art}

For a (smooth) function  $E(x_1,...,x_n)$, the Hessian ${\rm d}^2 E= \left( \frac{\partial^2  E}{\partial  x_i \partial x_j}\right)$ is a symmetric  bilinear form.  If it is positive definite, it defines a Riemannian metric called the {\it Hessian metric}.
Though the  construction
strongly depends on the coordinate system, Hessian  metrics naturally appear in many subjects of mathematics.

For example, for toric K\"ahler manifolds,  the metrics on the quotient space are  (locally)  Hessian metrics.
Metrics admitting nontrivial geodesic equivalence  are also   Hessian metrics, see e.g. \cite[\S 4.2]{BMR}. There is a strong relation between Hessian metrics and   the Hamiltonian construction in the theory of infinite-dimensional    integrable system of hydrodynamic type, see e.g.
\cite{GD}. Hessian  metrics  naturally come in many geometric  constructions of  Riemannian metrics inside convex domains
(see e.g. \cite{Yau}),
in affine geometry of hypersurfaces (see e.g. \cite{La1967,simon}) and
in information geometry (see e.g.  \cite{Sh2013}). We refer to   \cite{Sh2007} for a comprehensive study of differential geometry of Hessian metrics and their applications.

We are  interested in Hessian metrics that naturally appear in convex and Finsler geometry.
 They are defined on $\mathbb{R}^n \setminus{\{0\}}$ and  the function  $E$
satisfies the  following restriction:
 it is {\it positively 2-homogeneous}, that is, for any $\lambda{\geq0}$  we have
 $E(\lambda y)= \lambda^2 E(y)$.

 Under this assumption,  the property that ${\rm d}^2 E$  is positive definite  is  equivalent
 to the condition that
$E$ is positive on $\mathbb{R}^n\setminus \{0\}$ and that   $F:= \sqrt{ 2 E}$ satisfies the following properties:  it is  positively 1-homogeneous (i.e.,  $F(\lambda y)= \lambda F(y)$ for $\lambda\ge 0$), convex (i.e. $F( y_1 + y_2) \le
 F( y_1) + F( y_2)$)  and  strongly convex (i.e., the  second fundamental form  of the  {\it indicatrix}
  $S_F:= \{ y \in \mathbb{R}^n \mid  F(y)=1\}$ is positive definite).
Functions $F$ with such properties are called
{\it Minkowski norms}. All Minkowki norms we consider below are smooth and strongly convex.

It is known that  the indicatrix $S_F$ determines  the Minkowski norm  $F$ and
(as  we recall below)  that the  Hessian metric  of $E= \tfrac{1}{2} F^2 $ determines the function $E$.  So the study of
strongly convex bodies with smooth boundary can be reduced to the study of Hessian metrics for $E= \tfrac{1}{2} F^2$ and in particular apply methods and results of  Riemannian geometry. We refer to \cite{La1967,Sc2013}   for more details on the interrelation between Hessian geometry and convex geometry.
In later discussion, we will reserve the notation $F$ for the Minkowski norm and $E= \frac{1}{2} F^2 $ for the function we use to build a Hessian metric.

The appearance of Hessian metrics in
 Finsler geometry is related to that in the  convex geometry.
Recall that a  {\it Finsler metric}  on a smooth manifold $M$  with    $\operatorname{dim} M > 1$
 is a continuous
function $F$ on  $TM$ such that it is smooth on the slit tangent bundle $TM \setminus\{ 0\} $ and such that   its restriction to
 each tangent space $T_pM$ is a Minkowski norm.  The corresponding Hessian metric $g$ is then a Riemannian metric on the
slit
tangent space $T_pM\setminus \{0\}$. It was  called
{\it  the fundamental tensor} by L. Berwald \cite{Berwald1941} and it naturally comes to  many geometric constructions in Finsler geometry.

In this paper we study  isometries between the Hessian metrics of Minkowski norms.
We call the diffeomorphism $\Phi:\mathbb{R}^n\backslash\{0\}
\rightarrow\mathbb{R}^n\backslash\{0\}$ a {\it Hessian isometry}
from $F_1$ to $F_2$, if it is an isometry between the Hessian metrics $g_1= {\rm d}^2 E_1= \tfrac{1}2 {\rm d}^2 (F_1^2)$ and
$g_2= {\rm d}^2 E_2= \tfrac{1}2 {\rm d}^2 (F_2^2)$. By {\it  local Hessian isometry} we understand a positively 1-homogeneous diffeomorphism between two conic domains that is isometry with respect to the restriction of the Hessian metrics to these domains. Here
the {\it positive 1-homogeneity} for the local Hassian isometry $\Phi$ is the property that
$\Phi(\lambda x)=\lambda \Phi(x)$ for any $\lambda>0$ and any $x\in\mathbb{R}^n\backslash\{0\}$ where $\Phi$ is defined.
By  {\it conic} domain
we understand
$$
C(U):= \{\lambda  y \ \mid \ y \in U \ ,  \ \lambda >0\}, \textrm{\  where $U\subset \mathbb{R}^n\backslash\{0\}$}.
$$
Let us recall  some known facts (e.g. \cite{BCS,La1967}) that   follow from the positive 1-homogeneity of
 $F$.
\begin{itemize}
\item The Hessian metric determines  geometrically the ``radial''
rays, i.e., the sets  of the form $\{t y  \mid t\in \mathbb{R}_{>0}\}$, with nonzero $y$. Indeed, these rays are geodesics for the Hessian metrics, and are precisely those which are not complete.

\item The Hessian metric $g=\tfrac{1}{2} {\rm d}^2E$  determines  the functions $E$ and $F$ by $F(y)^2= g(y,y)$ for every $y\in \mathbb{R}^n\backslash\{0\}$.

\item  The Hessian metric $g$ is the cone metric over its  restriction  to the indicatrix $S_F$, i.e.,   $g=({\rm d}F)^2+F^2 g_{|S_F}$. That is,
 in any local coordinate system
$(r,\xi_2,...,\xi_{n})$ such that   $F(r,\xi_2,...,\xi_n)= r$, we have $g= dr^2 + r^2 \sum_{i,j=2}^{n} h_{ij}d\xi_id\xi_j$, where  the components $h_{ij}$ do not depend
on $r$.
\end{itemize}

These three observations imply that any  Hessian isometry $\Phi$ from $F_1$ to $F_2$ satisfies the positive 1-homogeneity
 and  diffeomorphically maps the indicatrix $S_{F_1}$ to $S_{F_2}$. Any  local  Hessian isometry
$\Phi:C(U_1)\to C(U_2)$ is 1-homogeneous by definition and diffeomorphically maps  $S_{F_1}\cap C(U_1)$ to  $S_{F_2}\cap C(U_2)$.

 Moreover, a positively 1-homogeneous mapping
 $\Phi$ which maps $S_{F_1}$ to $S_{F_2}$  is a
Hessian  isometry if an only if its restriction to $S_{F_1}$  an isometry between ${g_i}_{|S_{F_i}}$.

Let us now recall  some known examples of Hessian isometries.

If $\Phi:\mathbb{R}^n \to \mathbb{R}^n $ is a  linear isomorphism and $\phi^* F_2  =F_2\circ\Phi=F_1$, then $\Phi$ is trivially a Hessian
 isometry from $F_1$ to $F_2$. Indeed, for any linear coordinate change,
 the Hessian metric $g= \tfrac{1}{2}{\rm d}^2  (F^2)$ is covariant
 by the  Leibnitz formula. Such isometries will be   called  {\it linear isometries}.

Suppose dimension $n=2$. This case is completely understood and there are many examples of nonlinear Hessian isometries. To see this, let us
 consider the so-called  {\it generalised polar coordinates} on $\mathbb{R}^2 \setminus\{0\}$. This coordinate system is a special case of the
cone coordinate system discussed above.
It is  constructed as follows:
 the first coordinate is simply $F$,  so the indicatrix of $F$ is the coordinate line corresponding to the  value $1$.
    Next, on the indicatrix (which is a  closed convex simple curve) we denote by $\theta$ the arc-length  parameter corresponding to
		the Hessian metrics $g$.
		For each $y= (x_1,x_2)\in \mathbb{R}^2  \setminus \{0\}$, its $\theta$-coordinate is that
for $\tfrac{1}{F(y)} y\in S_F$.
See Fig. \ref{fig:1}.
\begin{figure}
\includegraphics[width=.4\textwidth]{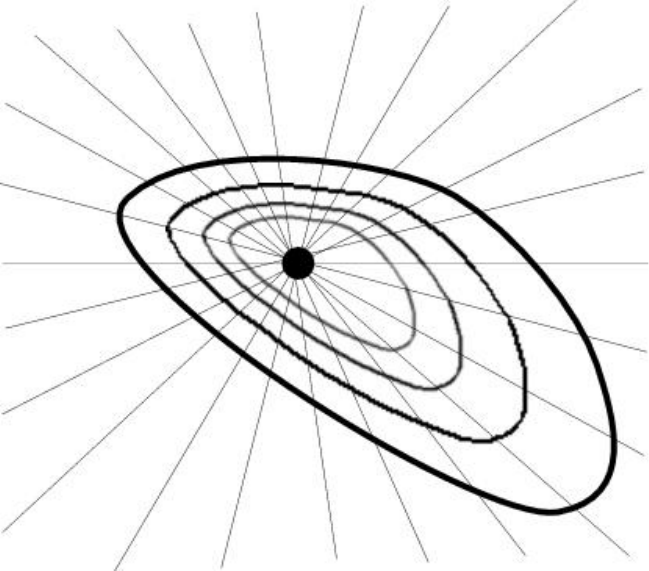}
\caption{Generalised polar coordinates $(F, \theta)$: First coordinate lines
are $\{F= \operatorname{const}\}$,  second coordinate lines are rays from zero. The second coordinate is chosen such that on $\{F=1\}$ it corresponds to the $g$-arclength parameter. } \label{fig:1}
\end{figure}

If
$E=\tfrac{1}{2} (x_1^2 + x_2^2)$  (so that  $g= dx_1^2 + dx_2^2$),  generalised polar coordinates
 are the usual  polar coordinates.  In the general case,  $\theta$ is still periodic and is defined
 up to addition of a constant to $\theta$ and the change of the sign, but the period  is not necessary $2\pi$.

  In the generalised polar  coordinates, the Hessian metric $g= \tfrac{1}{2}{\rm d}^2 (F^2)= {\rm d}F^2 + F^2  d\theta^2
$ is flat. So
we see that any two 2-dimensional  Minkowski  norms are locally  Hessian-isometric, and are Hessian-isometric if and only if their
 indicatrices  have the same  length in the corresponding Hessian metrics.

Let us now  consider $n\ge 3$. This case is almost completely open: in the literature
we found   one nonlinear    example of Hessian isometry, which we will recall and   generalise later,
and one negative result, which is the following Theorem:

\begin{theorem}[\cite{Br1967},  for alternative proof see \cite{Sc1968}] \label{thm:Brickell}
Let $F$  be a Minkowski norm on $\mathbb{R}^n$, $n\ge 3$. Assume it is \emph{absolutely homogeneous}, that is $F(\lambda y)=|\lambda|\cdot F(y)$ for every $\lambda\in\mathbb{R}$ and $y\in\mathbb{R}^n$.

Then, if  the Hessian metric $g= \tfrac{1}{2}{\rm d}^2 (F^2)$ on $\mathbb{R}^n\backslash\{0\}$ has zero  curvature,
 $F$ is \emph{Euclidean}, that is, $F= {\sqrt{ \sum_{i,j} \alpha_{ij} x_i x_j}}$ for a positive definite symmetric matrix $(\alpha_{ij})\in \mathbb{R}^{n\times n}$.  In this case,  every Hessian  isometry  is  linear.
\end{theorem}
The proofs in \cite{Br1967,Sc1968} are  different,  but the
 assumption that $F$   is absolutely homogeneous is essential  for both.

Let us now recall and slightly generalise the only known  example of nonlinear Hessian isometry in dimension $n\ge 3$.
We start with any Minkowski norm $F$ on the space $\mathbb{R}^n$ of column vectors, set $E=\tfrac{1}{2}F^2$ and   consider the corresponding
{\it Legendre transformation}:
 \begin{equation}\label{LT-original}
\Phi: \mathbb{R}^n\setminus \{0\}
 \to \mathbb{R}^n\setminus\{0\}\ , \ \ y= (x_1,...,x_n)^T\mapsto
 \left( \tfrac{\partial}{\partial  x_1}E(y),..., \tfrac{\partial}{\partial  x_n}E(y)\right)^T.
 \end{equation}
For the Euclidean Minkowski norm $F=\sqrt{x_1^2+...+x_n^2} $, the Legendre transformation $\Phi= \textrm{id}$.

Obviously
the function $\hat E=\Phi_*(E)$ on $\mathbb{R}^n\backslash \{0\}$
is a positive smooth function satisfying the positive 2-homogeneity. As we explain below in Remark \ref{rem:schneider} (see also \cite[\S 4.8]{BCS}),
the Hessian  of $\hat E$ is given by the matrix inverse to that for $g$ and is therefore positive definite.   Then,  $ \hat F= \sqrt{ 2\hat E}$ is a Minkowski norm.

  In  \cite{Sc2013} it was proved that  the Legendre transformation $\Phi$ in (\ref{LT-original})
is  a Hessian isometry from $F$ to $\hat F$.
Clearly, it is linear if and only if $F$ is Euclidean.


\begin{remark} \label{rem:schneider} R. Schneider's observation  that the Legendre transformation $\Phi$ in (\ref{LT-original}) is a Hessian isometry is important  for our  paper, so let us sketch a proof. Using $g_{ij}=\tfrac{\partial^2 E}{\partial x_i\partial x_j}$ for the Hessian metric of  $F$ and the explained above formula $E= \tfrac{1}{2} \sum_{i,j}g_{ij} x_i x_j$, the Legendre transformation $\Phi$ in (\ref{LT-original}) can be presented
at $y=(x_1,\cdots,x_n)^T\in\mathbb{R}^n\backslash\{0\}$ as
(see e.g. \cite[Eq. (14.8.1)]{BCS})
\begin{equation*} \label{eq:Phi1}
\Phi(y)=\left(\tfrac{\partial}{\partial x_j}\tfrac{1}{2}\sum_{i,k} g_{ik} x_i x_k\right)_{1\leq j\leq n}=\left(\sum_{i}    g_{ij}  x_i + \tfrac{1}{2}\sum_{i, k} \tfrac{\partial g_{ik}}{\partial x_j}x_ix_k \right)_{1\leq j\leq n}=\left(\sum_{i}    g_{ij}  x_i  \right)_{1\leq j\leq n}.
\end{equation*}
Here we have used $\sum_i\tfrac{\partial g_{ik}}{\partial x_j}x_ix_k=x_k\left(\sum_i\tfrac{\partial g_{ik}}{\partial x_j}x_i\right)=0$ by the positive $2$-homogeneity of $E$.
Then, its differential $d\Phi$  at $y$ has the Jacobi matrix
\begin{equation*} \label{eq:dphi}
\left(d\Phi \right)_{1\leq i,j\leq n} =    \left( g_{ij}  +  \sum_{k}  \tfrac{\partial g_{ik}}{\partial x_j} x_k \right)_{1\leq i,j\leq n}  = \left(g_{ij}\right)_{1\leq i,j\leq n}.
\end{equation*}
Since  the Legendre transformation is an involution,  $\Phi^{-1}$ is the Legendre transformation in (\ref{LT-original}) with $F=\hat{\hat{F}}$ and $\hat F$ exchanged,
the Hessian metric $\hat g$ for $\hat F$ at $\Phi(x)$ is represented by the inverse matrix $(g^{ij})_{1\leq i,j\leq n}$ for $F$ at $x$ (see  \cite[Proposition 14.8.1]{BCS} for more details). So the pullback $\Phi^* \hat{g}$ is given by the matrix
$$
 \left(\sum_{s,r}    g^{sr} g_{si} g_{rj}\right)_{1\leq i,j\leq n} =\left( g_{ij}\right)_{1\leq i,j\leq n},
$$
which  is the matrix of the  Hessian metric $g$ for $F$.
\end{remark}

%

Let us    now  modify   the above   example. We start with the Euclidean Minkowski norm $F_0= \sqrt{x_1^2+\cdots+x_n^2}$, and slightly deform it on two conic open subsets $C(U_1)$ and $C(U_2)$ of $\mathbb{R}^n\backslash\{0\}$, where
$U_1$ and $U_2$ are two open subsets of $S_{F_0}$ with disjoint closures. We obtain  a new Minkowski norm $F_1$.
Denote by  $\Phi$ the Legendre transformation and  by  $\hat{F}_1$ the push-forward of $F_1$.  See Fig. \ref{fig:2}. The second new Minkowski norm $F_2$ is constructed as follows: it coincides with
$\hat{F}_1$ on $C(U_2)$ and with $F_1$ on $\mathbb{R}^n\backslash C(U_2)$.  It is still  a smooth strictly convex Minkowski norm.
Next, we consider the mapping $\tilde \Phi$ such that it is   identity  on $C(U_1)$ and $\Phi$ on $\mathbb{R}^n\backslash C(U_1)$. It is
a Hessian isometry  from $F_1$ to $F_2$. If  $F_1$  is different  from    $F_0$ on both $C(U_1)$ and $C(U_2)$, $\tilde{\Phi}$ is
neither a linear isometry nor a Legendre transform.

\begin{figure}
\includegraphics[width=\textwidth]{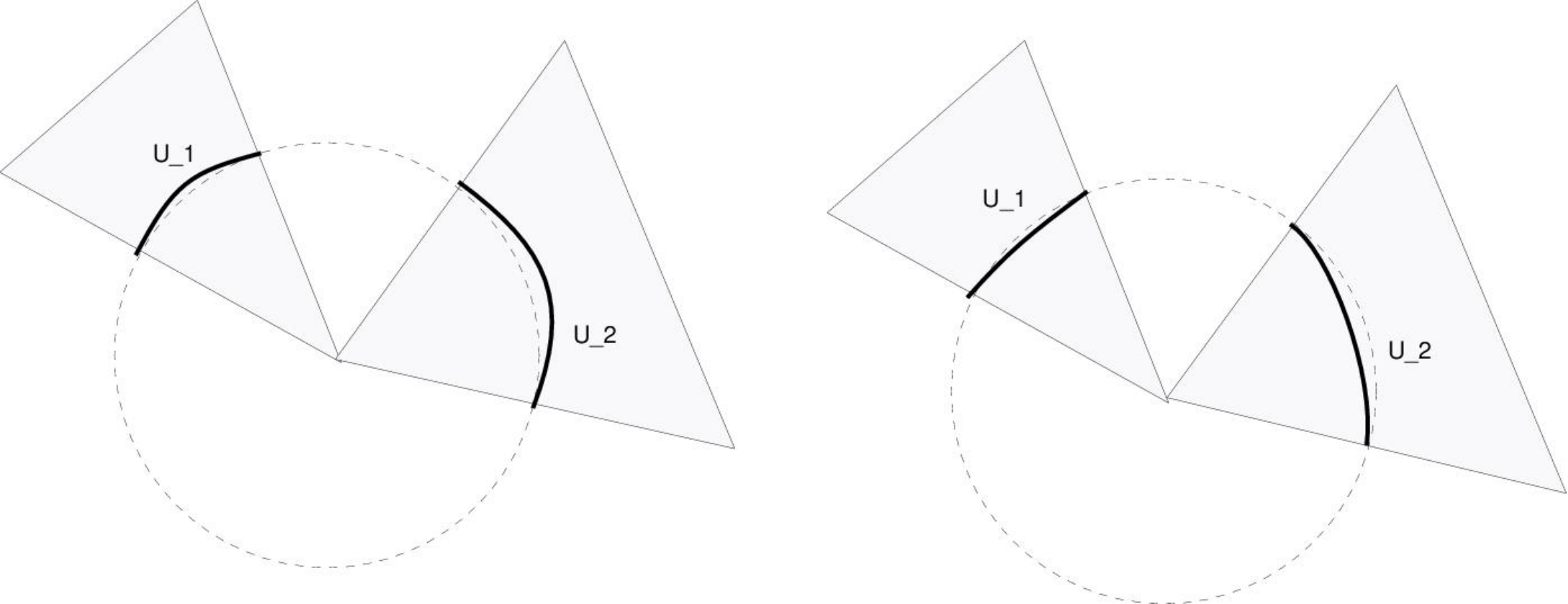}
\caption{Construction
 of  nonlinear and nonlegendre Hessian isometry: the functions $E_1=E$ is different from $x_1^2+...+x_n^2$ in cones over
$U_1$ and $U_2$ (grey triangles). The function $\hat E$  (second picture) is the Legendre-transform of $E=E_1$. The function $E_2$
coincides with $E_1$ everywhere but in $C(U_2)$ and in $C(U_2)$  it coincides with $\hat E$  } \label{fig:2}
\end{figure}

One can build this example such that $F_1$ and $F_2$  are preserved by  the standard blockdiagonal action of   $O(k)\times O(n-k)$ (of course in this case the  conic open sets $C(U_i)$ must be $O(k)\times O(n-k)$-invariant).  One can  impose additional symmetries on the construction so the resulting metric
$F_2$  has, in addition to this linear $O(k)\times O(n-k)$-symmetry, a nonlinear Hessian  self-isometry.
One can further generalise this example by starting with $F_0$ which is  not Euclidean  but still has `Euclidean pieces' and  by  deforming $F_0$ in
more than two (even infinitely many) open subsets.

%
\label{subsection-1-1}
\subsection{Results}

We consider a Minkowski norm $F$ on $\mathbb{R}^n$ with $n\geq 3$
which has a linear $SO(k)\times SO(n-k)$-symmetry, and study {\it connected isometry group} (i.e., the identity component of the group of all isometries) of the Hessian metric of $F$. We prove:

\begin{theorem}  \label{thm:main1}
Suppose $F$ is a Minkowski norm on $\mathbb{R}^n$ with $n\geq 3$, which is  invariant with respect to the standard block diagonal action of the group $SO(k)\times SO(n-k)$ with $1\leq k\leq n-1$. Let
$G_0$ be the connected isometry group for the Hessian metric $g=\tfrac12{\rm d}^2 F^2$ on $\mathbb{R}^n\backslash\{0\}$.

Then, every element $\Phi \in G_0$ is linear. Moreover, if $F$ is not  Euclidean,  then  $G_0$ together with  its action coincides with $SO(k)\times SO(n-k)$.   \end{theorem}

In Theorem \ref{thm:main1},
the {\it  standard block diagonal
$SO(k)\times SO(n-k)$-action} is the left multiplication on column vectors  by all block diagonal matrices $\mathrm{diag}(A',A'')$ with
$A'\in SO(k)$ and $A''\in SO(n-k)$.

Theorem \ref{thm:main1} is sharp in the following sense:
\begin{itemize} \item By an $SO(k)\times SO(n-k)$-equivariant modification for the Legendre transformation we have discussed at the end of Section \ref{subsection-1-1}, we can construct some nonlinear Hessian isometry $\tilde{\Phi}$. So $G_0$ in Theorem \ref{thm:main1}
can not be changed to the full group $G$ of all Hessian isometries on $(\mathbb{R}^n,F)$.
\item If $F$ is Euclidean, i.e., $F= \sqrt{\sum_{i,j}\alpha_{ij} x_i x_j}$ for some positive definite symmetric matrix $(\alpha_{ij})$, its Hessian metric $g$ on $\mathbb{R}^n\backslash\{0\}$
    is the restriction of a flat metric on $\mathbb{R}^n$. In this case  the group of all Hessian isometries is $O(n)$ and the
		connected isometry group $G_0$  is $SO(n)$.
	\item Theorem \ref{thm:main1} is not true locally. In Remark \ref{Rem:0040} we will show  the existence of  (smooth positively 1-homogeneous  strongly convex $SO(2)$-invariant)
	functions $F$ defined on a conic open subset of $\mathbb{R}^3\backslash\{0\}$ such that it is not Euclidean but the
	corresponding Hessian metric  is flat.
See also discussion in Section \ref{subsection-1-5}.
	
	\item Theorem \ref{thm:main1}  and also other results of our paper  trivially  hold  when $k=0$ or $k=n$, since in this case the Minkowski norm is automatically Euclidean.  In the proofs we assume without loss of generality $1\leq k\leq n/2$.
\end{itemize}

 Theorem \ref{thm:main1} implies that for any two non-Euclidean  Minkowski norms $F_1$ and $F_2$ which are
 invariant with respect to the standard blockdiagonal action of the group $SO(k)\times SO(n-k)$, with $n\geq 3$ and $1\leq k\leq n-1$,
    a Hessian isometry $\Phi$ from $F_1$ to $F_2$
    must map orbits to orbits (i.e., $\Phi$ maps each $SO(k)\times SO(n-k)$-orbit to an $SO(k)\times SO(n-k)$-orbit).

Next, we consider two Minkowski norms $F_1$ and $F_2$ on $\mathbb{R}^n$ which are invariant for the standard block diagonal action
of $SO(k)\times SO(n-k)$,
and study {\it local Hessian isometry which maps orbits to orbits}.
That means the local Hessian isometry $\Phi$ from $F_1$ to $F_2$ is defined between two $SO(k)\times SO(n-k)$-invariant conic open sets, $C(U_1)$ and $C(U_2)$,
under the additional assumption that $\Phi$ maps each $SO(k)\times SO(n-k)$-orbit in $C(U_1)$  to  that in $C(U_2)$.
%

\begin{theorem} \label{thm:mainA} Let $F_1$   be a
 Minkowski norm  on $\mathbb{R}^n $  which  is   invariant for the standard block diagonal $SO(k)\times SO(n-k)$-action, with $n\geq 3$ and $1\leq k\leq n-1$.
Assume $C(U_1)$ is an $SO(k)\times SO(n-k)$-invariant connected conic open subset of $\mathbb{R}^n\backslash\{0\}$, such that
 every $y\in C(U_1)$ satisfies
\begin{equation}\label{eq:notzero}
g_1(v',v'')\neq 0,\quad\mbox{for some }v'\in V'\mbox{ and } v''\in V''.
\end{equation}
Here $g_1=g_1(\cdot,\cdot)$ is the Hessian metric of $F_1$, and
$\mathbb{R}^n=V'\oplus V''$ is an $SO(k)\times SO(n-k)$-invariant
decomposition with $\dim V'=k$ and $\dim V''=n-k$.

Then, for any $SO(k)\times SO(n-k)$-invariant Minkowski norm $F_2$, and any local Hessian isometry $\Phi$
from $F_1$ to $F_2$ which is defined on $C(U_1)$ and maps orbits to orbits, $\Phi$ either  coincides with the restriction of  a linear isometry, or it coincides with the restriction of  the composition of   the $F_1$-Legendre transformation  and a linear isometry.
\end{theorem}

Let us emphasize that near  the points such that \eqref{eq:notzero} holds the Minkowski norm $F$ is not Euclidean so the $F_1$-Legendre transformation is not linear. In particular, $\Phi$ can not be  simultaneously  linear and the  composition of the   the $F_1$-Legendre transformation  and a linear isometry.

The condition (\ref{eq:notzero}) in Theorem \ref{thm:mainA}
characterizes one class of generic points on $S_{F_1}$ where $S_{F_1}$ does not touch any $O(k)\times O(n-k)$-invariant ellipsoid with an order bigger than one.  Of course, (\ref{eq:notzero}) is an open condition. But still the set of the points such that (\ref{eq:notzero})  is not fulfilled (for all $v'$ and $v''$) may contain nonempty open subset.  We discuss  such open domains in the following  Theorem:

\begin{theorem}\label{thm:mainB}
Let $F_1$  be a
 Minkowski norm on $\mathbb{R}^n $  which is   invariant for the standard block diagonal $SO(k)\times SO(n-k)$-action, with $n\geq 3$ and $1\leq k\leq n-1$.  Assume $C(U_1)$ is an $SO(k)\times SO(n-k)$-invariant connected conic open subset of $(\mathbb{R}^n\backslash\{0\},g_1)$ such that  at every $y\in C(U_1)$
\begin{equation}\label{eq:zero}
g_1(v',v'')= 0,\quad\mbox{for all }v'\in V'\mbox{ and } v''\in V''.
\end{equation}
Here $g_1=g_1(\cdot,\cdot)$ is the Hessian metric of $F_1$, and
$\mathbb{R}^n=V'\oplus V''$ is an $SO(k)\times SO(n-
k)$-invariant
decomposition with $\dim V'=k$ and $\dim V''=n-k$.

Then the restriction of $F_1$ to $C(U_1)$ is Euclidean.
Moreover,  for any $SO(k)\times SO(n-k)$-invariant Minkowski norm $F_2$, and any local Hessian isometry $\Phi$
from $F_1$ to $F_2$ which is defined on $C(U_1)$ and maps orbits to orbits, we have that $\Phi$ coincides with the restriction of a linear isometry and that the restriction of $F_2$ to $C(U_2)=\Phi(C(U_1))$ is Euclidean.
\end{theorem}

The example discussed in Remark \ref{Rem:0040} shows that the condition that $\Phi$ maps orbits to orbits is necessary for Theorem
\ref{thm:mainB}.

Theorem \ref{thm:mainA} and Theorem \ref{thm:mainB}
provide the precise and explicit description for a local (or global) Hessian isometry $\Phi$ almost everywhere in its domain. We can
find two $SO(k)\times SO(n-k)$-invariant conic open subsets $ C({U}')$ and $ C({U}'')$ in $\mathbb{R}^n\backslash\{0\}$, such that $ C({U}')\cup C({U}'')$ is dense in the domain of $\Phi$,
 (\ref{eq:notzero}) is satisfied on $ C({U}')$, and (\ref{eq:zero}) is satisfied on $ C({U}'')$. Then by these two theorems, when restricted to each connected component $C(U'_1)$ of $C(U')$, $\Phi$ is a linear isometry or the composition of the Legendre transformation of $F_1$  which we denote by $\Psi$ and a linear isometry. Restricted to each connected component of $C(U'')$, $\Phi$ is a linear isometry. This implies that every such  $\Phi$ can be constructed   along the lines
 discussed at the end of  Section \ref{subsection-1-1}.

\label{section_results}

\subsection{Applications in convex geometry: a special case of  Laugwitz Conjecture. }

It was conjectured by   D. Laugwitz \cite[page 70]{La1967} that Theorem \ref{thm:Brickell} remains true without the assumption of absolute homogeneity:
\begin{conjecture}[Laugwitz Conjecture]
If the Hessian metric $g=\tfrac12{\rm d}^2F^2$ for a
Minkowski norm $F$ is flat on $\mathbb{R}^n\backslash\{0\}$ with $n\geq3$, then $F$ is Euclidean.
\end{conjecture}
For a discussion from the viewpoint of  Finsler geometry  see e.g.
	\cite[Remark (b) on page 416]{BCS}. Using Theorem \ref{thm:main1}, we prove the following special case of Laugwitz Conjecture.
	
	\begin{corollary} \label{cor:1}
	Laugwitz conjecture is true for the class of Minkowski norms which are invariant with respect to the standard block diagonal
$SO(n-1)$-action.
	\end{corollary}

Indeed,  if the Hessian metric of $F$ is flat on $\mathbb{R}^n\backslash\{0\}$, then the identity component $G_0$ of all Hessian isometries for $F$ has the
dimension $\tfrac{n(n-1)}{2}$. As a Lie group, $G_0$ is isomorphic to $SO(n)$, but  its action on $\mathbb{R}^n$ is linear iff $F$ is Euclidean. Since we have assumed here that $F$ is invariant with respect to the standard block diagonal action of $SO(n-1)=SO(1)\times SO(n-1)$ with $n\geq 3$, and obviously $G_0=SO(n)$ has a bigger dimension than $SO(n-1)$,
the
last statement in Theorem \ref{thm:main1} for $k=1$ or $k=n-1$ guarantees that the $G_0$-action is linear in this case.

By similar argument, it follows from
Theorem \ref{thm:main1} that the Laugwitz conjecture is true
for Minkowski norms which are invariant for the standard block diagonal $SO(k)\times SO(n-k)$-action with $2\leq k\leq n-2$.
Notice that it has already been covered by Theorem \ref{thm:Brickell}, because the norms are absolutely homogeneous in this case.
\label{subsection-1-3}
 \subsection{Application in Finsler geometry: a special case of
Landsberg Unicorn Conjecture} \label{sec:intro2}

Historically   Finsler geometry    appeared as an attempt of   generalising    results and methods
  from Riemannian geometry to  the optimal transport and calculus of variation, see e.g. \cite{ Berwald1941, Bl1906, Cartan, Ha1903,Landsberg1908, Rund}.
		Generalisation of Riemannian results to the Finslerian setup is   still  one
  of the most popular research directions    in Finsler geometry, and one of the main sources for interesting problems and
	methods.
			
			The analogs of Riemannian objects in  Finsler geometry  are in many cases  more complicated than Riemannian originals \cite{Sh2001}.  The connection (actually, there are three main natural candidates for the generalisation of the  Levi-Civita connection) is generically not linear. It results in the nonlinearity for the {Berwald parallel transport}, which will be addressed later. The  analogs of
		the  Riemannian 	curvatures  are also  more complicated  and in fact there exist  two main   different types of the curvature: the Riemannian type and the non-Riemannian type. For example, the \emph{flag curvature}, which generalizes the sectional curvature in Riemannian geometry, is of the Riemannian type.  On the other hand, the {\it Landsberg curvature} is of the non-Riemannian type, because it vanishes identically for Riemannian metrics and  has no analogs in  Riemannian geometry.

It is known  that the  Landsberg curvature vanishes identically for a relatively small   class of Finsler metrics  called {\it Berwald metrics}, which are characterized by the property that the Berwald parallel transport is linear, see e.g. \cite[Proposition 4.3.2]{CS2005} or \cite[\S 10]{BCS}. Berwald metrics are completely understood, see e.g. \cite[Theorem 4.3.4]{CS2005}, \cite[\S\S 8,9]{MT2012} or \cite{Sz1981}.

A non-Berwald Finsler metric with vanishing Landsberg curvature is
called a {\it unicorn metric}. Many experts believe that smooth unicorn metrics do not exist. This statement  is called the  {\it Landsberg Unicorn Conjecture}.

 \begin{conjecture}[Landsberg Unicorn Conjecture]
 \label{Landsberg-Conjecture}
 A  Finsler metric with vanishing Landsberg curvature must be Berwald. \end{conjecture}

 The origin of this conjecture can be traced back to  \cite{Berwald1947} (or  even to  \cite{Landsberg1908}).  It is  definitely one of the   most popular  open problems  in Finsler geometry   and was explicitly asked in   e.g.   		\cite{AlvarezPaiva2006, Bao2007, BaoChernShen1997,Dodson2006, Ma1996, Shen2009a}.
Its proof was reported a few times in preprints and even published in reasonable journals, but  later crucial mistakes were found, see e.g. \cite{Matveev2009b}.

The definition of the Landsberg curvature  and the properties of Finsler metrics with vanishing Landsberg curvature can be found elsewhere, e.g. in  \cite[\S 2.1 and \S 4.4]{CS2005}. For our paper, we only need the following  known statement:

\begin{fact}[e.g. Proposition 4.4.1 of \cite{CS2005} or \cite{Kozma}] \label{fact:1}
If Landsberg curvature vanishes, then  the Berwald parallel transport is isometric with respect to the Hessian metric (corresponding to $E= \tfrac{1}{2} F^2$ in each tangent space).
\end{fact}

Recall that the {\it Berwald parallel transport} is a Finslerian  analog of the parallel transport in Riemannian geometry. For every smooth curve $c:[0,1]\to M$  on $(M,F)$, the  Berwald parallel transport along $c$ provides a smooth family of diffeomorphisms $\Phi_s:T_{c(0)}M\backslash\{0\}\to T_{c(s)}M\backslash\{0\}$.  Similarly to the Riemannian case, the mapping is defined via
 certain  system of ODEs along the curve $c$. Differently from the Riemannian case, these ODEs  are  not linear, so for a generic  Finsler metric the Berwald
 parallel transport is not linear as well.  In fact, as recalled above, it is linear if and only if the metric is Berwald.

In Section \ref{sec:prooflandsberg} we explain that Theorems \ref{thm:main1}, \ref{thm:mainA} and  \ref{thm:mainB} easily imply the following important special case of  Conjecture \ref{Landsberg-Conjecture}.

\begin{corollary} \label{thm:main2} Let $(M,F)$ be  a  Finsler
 manifold of dimension $n\ge 3$. Assume that
 for every point $p\in M$, there exist linear coordinates in  $T_pM$  such that
the restriction $F_{|T_pM}$ is invariant with respect to the standard block diagonal action of the group $SO(k)\times SO(n-k)$ with $1\leq k\leq n-1$.

Then, if the Landsberg curvature  vanishes,  $F$ is Berwald.
\end{corollary}

Many special cases of Corollary \ref{thm:main2} appeared in the literature before. Let us give some examples with the dimension $n\geq3$:  \cite{Ma1974} (see also \cite{MZ2002}) proved that every
Randers metric such that its Landsberg curvature is zero is Berwald.
 \cite{Sh2009} proved that every $(\alpha,\beta)$ metric with zero Landsberg curvature
 is Berwald. \cite{ZWL2019} proved that every general $(\alpha, \beta)$ metric
with zero Landsberg curvature is Berwald. All these  results  follow  from Corollary \ref{thm:main2} with $k=1$, since  for every $p\in M$ the restriction of a Randers, $(\alpha,\beta)$ or general $(\alpha, \beta)$ metric to $T_pM$ is invariant with respect to a block diagonal action of $SO(n-1)$ \cite{DX2015}. Indeed, general  $(\alpha, \beta)$
is defined as follows: one takes a  Riemannian metric $\alpha=(a_{ij})$, a  $1$-form  $\beta=(\beta_i)$,
  a function $\varphi$ of two varables, and defines  $F$ by the formula
\begin{equation}\label{eq:albet}
F(p,y)=\varphi\left( |\beta|_{\alpha},\frac{\beta(y)}{\sqrt{\alpha(y,y)}}\right) \sqrt{\alpha(y,y)},
\end{equation}
where $|\beta|_{\alpha}= \sqrt{\alpha^{ij}\beta_i\beta_j}$  is the point-wise norm of $\beta$ in $\alpha$ and $\alpha(y,y)= \alpha_{ij}y^iy^j =
\left(|y|_{\alpha}\right)^2$. The function    $\varphi$  is
 chosen such that \eqref{eq:albet} is a Finsler metric. For certain $\varphi$, additional restrictions on $|\beta|_{\alpha}$ must be assumed to insure  the result is a Finsler metric.
 chosen such that \eqref{eq:albet} is a Finsler metric. For certain $\varphi$, additional restrictions on $|\beta|_{\alpha}$ must be assumed to insure  the result is a Finsler metric.

$(\alpha, \beta)$ metrics are  general $(\alpha, \beta) $ metrics such that the function
 $\varphi$ does not depend on $|\beta|_{\alpha}$ (so it is a function of one variable). Randers  metrics are $(\alpha, \beta)$ metrics for the function  $\varphi(t)= 1+ \tfrac{1}{t}$. In the last
 case the restriction insuring  that this $\varphi$ determines a Finsler metric is $|\beta|_{\alpha}<1$.

Note that  the proofs from \cite{Ma1974,Sh2009,ZWL2019} essentially use that the function $\varphi(t,s)$ is the same at all points of the manifold, so
 the dependence of Randers, $(\alpha, \beta)$ and general   $(\alpha, \beta)$ metrics   on the position $p\in M$  essentially goes
 through the  dependence of $\alpha$ and $\beta$ on $p$ only.
 In our proof we need only that in each tangent space $F$ has a linear  $SO(n-1)$-symmetry. In other words, the function $\varphi$ may arbitrary depend on the point $p$ of the manifold.

Another  example of such type is \cite{DX2016, XD2019}: there, the so-called  $(\alpha_1,\alpha_2)$ metrics are considered, their definition which we do not recall here is similar to that of $(\alpha, \beta)$ metrics.
 In this case, the restriction of the metric to each tangent space is invariant with respect to the $SO(k)\times SO(n-k)$-action. The analog of the function $\varphi$ is the same at all points of the manifold so
the dependence of the  metric on position goes through  $\alpha_1$ and $\alpha_2$ only. By our result, the function $\varphi$ may
 arbitrarily depend on the position.

A slightly different result which also follows from  Corollary \ref{thm:main2} is in \cite{MZ2014}, where
nonexistence of non-Berwaldian  Finsler manifolds with vanishing Landsberg curvature was shown in  the class of spherically symmetric metrics. By definition,
 Finlser metric on $\mathbb{R}^n\setminus\{0\}$ is {\it spherically symmetric}, if it
is invariant with respect to the standard action of $SO(n)$. This condition  implies that
 the restriction of $F$ to  every tangent space has $SO(n-1)$-symmetry and  Corollary \ref{thm:main2} is applicable.

Alternative   geometric approach  that was successfully used for the proof of  Landsberg Unicorn Conjecture for certain generalisations of
$(\alpha,\beta)$ metrics is based on  semi-C-reducibility \cite{Cr2020, FL2018, MS1979}.   The results of these papers related to the Landsberg Unicorn Conjecture  also easily follow from our Corollary \ref{thm:main2}. Notice that generic $(\alpha_1,\alpha_2)$ metrics do not satisfy the semi-C-reducibility.

\subsection{  Smoothness assumption is necessary.}
\label{subsection-1-5}

G. Asanov constructed some singular norms $F$ on $\mathbb{R}^3$ with the standard $SO(2)$-symmetry \cite{An1995,An1998}
His examples  can be generalised to any dimension $n \ge 3$ and give singular norms on    $\mathbb{R}^n$ with linear $SO(n-1)$-symmetry, see e.g. \cite{ZWL2019}. They  lead to the construction of
first singular unicorn metrics \cite{As2006,As2007} and were actively discussed in the literature (e.g. \cite{Cr2011}).

The Minkowski norms in all these examples are not smooth at the line which is fixed by the $SO(n-1)$-action,  but they are  smooth and even real analytic elsewhere.  Their isometry group is $O(n-1)$ but locally the algebra of Killing vector fields is isomorphic to $so(n)$ and has the dimension $\tfrac{(n-1)n}{2}$.

Within this paper we assume that all objects we consider are sufficiently smooth.
Asanov's  examples and their generalisations
show that this smoothness
 assumption  is necessary.   It also shows (complimentary to Remark \ref{Rem:0040})  that Theorem \ref{thm:main1} is not a local statement.

\section{Hessian isometry on a Minkowski space with $SO(k)\times SO(n-k)$-symmetry}
\subsection{Setup.}

Within the whole section we  work in a Minkowski space $(\mathbb{R}^n,F)$ with $n\ge 3$. We denote $S_F=\{ y\in \mathbb{R}^n \mid  F(y)=1\}$ the indicatrix of $F$,
  and $g$ the Hessian metric $\tfrac{1}{2}{\rm d}^2F^2$ of $F$ on $\mathbb{R}^n\backslash\{0\}$ or its restriction to $S_F$ (and other submanifolds).  We assume that $F$ is invariant with respect  to the standard block diagonal action of
	$SO(k)\times  SO(n-k)$, with  $1\leq k\leq n/2$.
		
We start with the following simple observation:
		
\begin{lemma} \label{Lem:1}
Suppose $F$ is a Minkowski norm on $\mathbb{R}^n$ which is invariant with respect to the standard block diagonal action of
$SO(k)\times SO(n-k)$ with $n\geq 3$ and $1\leq k\leq n/2$.
Then $F$ is invariant with respect to the standard block diagonal action of $O(n-1)$ or $O(k)\times O(n-k)$ when $k=1$ or $k>1$ respectively.
\end{lemma}
		
Note that $SO(1) =\{e\}$   so the action
of $O(n-1)=SO(1)\times  O(n-1)$ is  just that by the orthogonal matrices of the form $\mathrm{diag}(1,A)$ with $A\in O(n-1)$.

\begin{proof} Clearly, when $k\neq 1$,
the orbits of the action of $SO(k)\times \{e\}$
 coincide with that of  $O(k)\times \{e\}$,
 so the function $F$, which is invariant with respect to the action of $SO(k)\times \{e\}$, is also invariant with respect to the action of $O(k)\times \{e\}$.
Similarly, by $k\leq n/2\leq n-2$, $F$ is invariant with respect to the
action of $\{e\}\times O(n-k)$.
\end{proof}

\subsection{Proof of Theorem \ref{thm:main1}  for $k=1$.}

We consider the indicatrix $S_F$ with the restriction of the Hessian metric $g$. Let $G_0$
be the connected isometry group for $(\mathbb{R}^n\backslash\{0\},g)$, then it is also the connected isometry group for $(S_F,g)$. We assume that $F$ is invariant with respect to the standard  block diagonal action of $SO(n-1)$. It implies that $G_0$
 naturally  contains the group $SO(n-1)$ as a subgroup.

If  $G_0$ coincides with $SO(n-1)$, there is nothing to prove.  The next Lemma shows that if $G_0$ does not coincide with $SO(n-1)$ then $(S_F, g)$
is isometric to the standard unit sphere.

\begin{lemma} \label{Lem:2}
In the notation above, assume $G_0$ does not coincide with $SO(n-1)$. Then  $(\mathbb{R}^n\backslash\{0\}, g)$ is flat, and $(S_F,g)$ has constant sectional curvature 1.
\end{lemma}
\begin{proof}
Let us assume that $G_0$ does not coincide with $SO(n-1)$, i.e.,
$\dim G_0\geq \tfrac{(n-1)(n-2)}{2}+1$.

We  first prove that $(S_F, g)$ is a homogeneous Riemannian sphere.
Here we apply a proof of this claim for all $n\geq 3$, which is similar to that of
 \cite[Theorem 1]{Ya1953}, see also \cite[\S 4]{Is1955}.
Notice that when $n\neq5$, \cite[Theorem 1]{Ya1953} provides an alternative approach. Indeed, we can also see that $(S_F,g)$ has constant sectional curvature, by \cite[Theorem  10]{Ob1955} and \cite[Theorem 5]{Is1955} when $n\neq 5$ and $n=5$ respectively, though it would not be needed in later argument.

Consider  the ``pole''
 $y_0= (a_0,0,...,0)\in S_F$. It  is a fixed point for the $SO(n-1)$-action. Consider its $G_0$-orbit
$$G_0\cdot y_0= \{\Phi(y_0)\mid \Phi\in G_0\}.$$
Let $H\subset G_0$ be
the
stabilizer  of  $y_0$.
 It is known that the stabilizer of a point with respect to an isometric action on an $(n-1)$-dimensional manifold
 is at most $\tfrac{(n-1)(n-2)}{2}$-dimensional, so we have
$\dim G_0>\tfrac{(n-1)(n-2)}{2}\geq \dim H$, i.e.,
there exists $y\in G_0\cdot y_0$ with $y\neq y_0$. The orbit $G_0\cdot y_0$ is connected, so we can find a curve $\gamma\subset G_0\cdot y_0$ connecting $y$ and $y_0$.
Then $G_0\cdot y_0\supset
SO(n-1)\cdot\gamma$ contains a $SO(n-1)$-invariant neighbourhood $U_0$ of $y_0$ in $S_F$. By its homogeneity,
$G_0\cdot y_0$ is an open subset of $S_F$. On the other hand, it is closed because $G_0$ is a compact Lie group. So we must have
$G_0\cdot y_0=S_F$, i.e., $(S_F,g)$ is a homogeneous sphere.

%
%
%
%
%

Next, we prove that the Hessian metric $g$ on $\mathbb{R}^n\backslash\{0\}$ is flat, and its restriction to  $S_F$
has constant curvature 1.

The {\it Cartan tensor} at $y=(x_1,\cdots,x_n)\in\mathbb{R}^n\backslash\{0\}$ is defined as
\begin{eqnarray*}
C(u,v,w)=\tfrac{1}{4}
\tfrac{\partial^3}{\partial s_1\partial s_2 \partial s_3}_{|s_1=s_2=s_3=0}F(y+s_1u+s_2v+s_3w)^2.
\end{eqnarray*}
for any $u,v,w\in\mathbb{R}^n=T_y\mathbb{R}^n$ (so its $(ijk)$-component is
$C_{ijk}=\tfrac{1}{4}\tfrac{\partial^3 (F^2)}{\partial x_i\partial x_j \partial x_k}$).

Now we show the Cartan tensor vanishes
 at $y_0=(a_0,0,\cdots,0)\in S_F$.

Clearly, it is multiple linear and totally symmetric.
By the positive
 $1$-homogeneity of $F$,
 at every point $y\in \mathbb{R}^n\setminus \{0\}$
 and for every vectors $u,v\in\mathbb{R}^n$,  we have
  $C(y, u, v)=0$ at $y$.
So we only need to show, for each vector
	$v$ with zero $x_1$-coordinate (i.e., $v\in T_{y_0}S_F$), we have $C(v, v, v)=0$ at $y_0$. Cartan's trick can be applied to avoid direct calculation. The group $SO(n-1)$ acts transitively
	on the unit $g$-sphere in $T_{y_0}S_F$. So there exists $A\in SO(n-1)$  with $Av=-v$. That means, the linear isometry induced by $A$ fixes $y_0$ and has a tangent map at $y_0$ mapping $v$ to $-v$.
It preserves the Cartan tensor as well, so we have
	$$C(v,v,v)= C(-v, -v, -v)=-C(v,v,v)$$
at $y_0$, which implies $C=0$ there.
		
Now we use the following well-known  fact in Hessian geometry:
	
\begin{fact}[e.g. Proposition 3.2 of \cite{Sh2007}]\label{lemma-curvature-formula-Hessian-geometry}
Consider the Hessian metric generated by a  (not necessary 2-homogeneous) function $E$, $g= {\rm d}^2  E$. Then, its curvature tensor $R_{ijk\ell}$
is given by
\begin{equation}\label{eq:fact-1}
R_{ijk\ell} = \tfrac{1}{4} \sum_{s,r} \left(\frac{\partial^3 E}{\partial x_j \partial x_\ell \partial x_s} g^{sr}  \frac{\partial^3 E}{\partial x_k \partial x_i \partial x_r}-  \frac{\partial^3 E}{\partial x_i \partial x_\ell \partial x_s} g^{sr}  \frac{\partial^3 E}{\partial x_k \partial x_j \partial x_r}\right),\end{equation}
where $g^{rs}$ denote the components of the matrix inverse to $(g_{rs})$.
\end{fact}	

If
$E= \tfrac{1}{2} F^2$ for a Minkowski norm $F$, the curvature formula \eqref{eq:fact-1}
 is reduced to
\begin{equation}\label{eq:fact}
R_{ijk\ell} = \sum_{s,r} \left(C_{i\ell s} g^{sr}  C_{jk r} -   C_{iks} g^{sr}  C_{j\ell r} \right).
\end{equation}
As we explained above, at $y_0$, every $C_{ijk}$ vanishes, so we have
$R_{ijk\ell} =0$, $\forall i,j,k,\ell$. In particular, the sectional curvature of $(\mathbb{R}^n\backslash\{0\},g)$ vanishes at $y_0$.

As we recalled in  Section \ref{subsection-1-1}, the Hessian metric $g=({\rm d}F)^2+F^2 g_{|S_F}$ on $\mathbb{R}^n\backslash\{0\}$ is the cone metric over  its restriction to $S_F$. Then by Gauss-Codazzi equation, the sectional curvature of $(S_F,g)$ equals to $1$ at $y_0$.   Since $(S_F,g)$ is
 homogeneous by assumptions,  $(S_F,g_{S_F})$ has constant sectional curvature 1 at every point, i.e., it is isometric to the standard unit sphere.
Then, the metric $g$  is flat as we claimed.
\end{proof}

The next Lemma  finishes the proof of Theorem \ref{thm:main1} for
$k=1$.

\begin{lemma} \label{lem:final}
Let $F$ be a Minkowski norm  on $\mathbb{R}^n$ with $n\geq 3$, which is
invariant with respect to the standard block diagonal action of $O(n-1)$. Assume the curvature of the Hessian metric $g= \tfrac{1}{2} {\rm d}^2(F^2)$ on $\mathbb{R}^n\backslash\{0\}$
identically vanishes. Then $F$ is Euclidean.
\end{lemma}

\begin{proof}
We  first prove Lemma \ref{lem:final} when $n=3$.

We consider the spherical  coordinates $(r, \theta, \phi)\in
\mathbb{R}_{>0}\times (0,\pi)\times (\mathbb{R}/(2\mathbb{Z}\pi))$  on $\mathbb{R}^3$ determined by
$$
x_1=r \cos\theta  , \  \ x_2= r \sin\theta \cos\phi, \   \ x_3= r  \sin\theta \sin\phi.$$
The $SO(2)$-action is the left multiplication on column vectors by matrices of the form
$$
\begin{pmatrix} 1 & 0 & 0\\
                0 & \cos s & \sin s\\
                 0 & -\sin s & \cos s\end{pmatrix},
$$
i.e., it fixes the $r$- and $\theta$-coordinates and shifts the $\phi$-coordinate.
By its $SO(2)$-invariancy and homogeneity,   the function
$E= \tfrac{1}{2}F^2$  can be presented as
\begin{equation} \label{eq:F}
E= r^2 f(\theta).
\end{equation}
By the symmetry $(x_1,x_2,x_3)\mapsto  (x_1, -x_2, -x_3)$ for $E$, the function $f(\theta)$ on $(0,\pi)$ can be extended
to and will be viewed as an even positive smooth function on $\mathbb{R}$ with the period $2\pi$, i.e., the restriction of
$E$ to the circle $\{(x_1,x_2,0)\mid x_1= \cos s, x_2= \sin s, \forall s\in\mathbb{R}\}$.

Let us now  calculate the Hessian metric $g$ and the Cartan tensor
$C$ of $F$ in the spherical coordinates. We use  subscripts and superscripts $r$, $\theta$ and $\phi$,  for example, $g_{r\theta}=g(\tfrac{\partial}{\partial r},\tfrac{\partial}{\partial \theta})$, and
$C_{\theta\theta\phi}=C(\tfrac{\partial}{\partial \theta},\tfrac{\partial}{\partial \theta},\tfrac{\partial}{\partial \phi})$.

By its definition, $g$
is the second covariant derivative of $E$
with respect to the Levi-Civita connection of the standard flat metric on $\mathbb{R}^3$, so we have
\begin{equation}\label{0000}
g(X,Y)=X(Y(E))-(\tilde{\nabla}_XY)(E),
\end{equation}
for any smooth tangent vector fields $X$ and $Y$ on $\mathbb{R}^3\backslash\{0\}$, where $\tilde{\nabla}$ is the Levi-Civita connection for the standard flat
metric
$$\tilde g:= {\rm d}x_1^2+ {\rm d}x_2^2 + {\rm d}x_3^2=
{\rm d}r^2+r^2 {\rm d}\theta^2+r^2\sin^2\theta {\rm d}\phi^2.$$
Direct calculation  gives
\begin{equation}\label{0001}
\begin{array}{ll}
\tilde{\nabla}_{\tfrac{\partial}{\partial r}}{\tfrac{\partial}{\partial r}}
=0, &
\tilde{\nabla}_{\tfrac{\partial}{\partial \theta}}{\tfrac{\partial}{\partial r}}
=\tilde{\nabla}_{\tfrac{\partial}{\partial r}}{\tfrac{\partial}{\partial \theta}}
=\tfrac{1}{r}\tfrac{\partial}{\partial\theta},\\
\tilde{\nabla}_{\tfrac{\partial}{\partial \phi}}{\tfrac{\partial}{\partial r}}
=
\tilde{\nabla}_{\tfrac{\partial}{\partial r}}{\tfrac{\partial}{\partial \phi}}
=\tfrac{1}{r}\tfrac{\partial}{\partial\phi}, &
\tilde{\nabla}_{\tfrac{\partial}{\partial \theta}}{\tfrac{\partial}{\partial \theta}}
=-r\tfrac{\partial}{\partial\phi},\\
\tilde{\nabla}_{\tfrac{\partial}{\partial \theta}}{\tfrac{\partial}{\partial \phi}}
=
\tilde{\nabla}_{\tfrac{\partial}{\partial
\phi}}{\tfrac{\partial}{\partial \theta}}=\tfrac{\cos\theta}{\sin\theta}
\tfrac{\partial}{\partial\phi},
&
\tilde{\nabla}_{\tfrac{\partial}{\partial \phi}}{
\tfrac{\partial}{\partial \phi}}
=-r\sin^2\theta\tfrac{\partial}{\partial r}-\sin\theta\cos\theta
\tfrac{\partial}{\partial\theta}.
\end{array}
\end{equation}
Combining \eqref{eq:F} and  (\ref{0000}), we  obtain all components $g_{ab}$, $a,b\in\{r,\theta,\phi\}$, for  $g={\rm d}^2 E$.  With the specified order
$(r,\theta,\phi)$, they can be presented as the following matrix,
\begin{equation} \label{eq:g}
\left[ \begin {array}{ccc}  2 f (\theta) &r{
\frac {\rm d}{{\rm d}\theta }}f (\theta) &0
\\ \noalign{\medskip}
r{\frac{\rm d}{{\rm d}\theta }}f (\theta) & 2r ^{2}f(\theta)+r^2
\frac{{\rm d}^{2}}{{\rm d}\theta^{2}}f (\theta)
&0\\ \noalign{\medskip}
0&0&2r^2\sin^2\theta f(\theta)+r^2\sin\theta\cos\theta
\frac{\rm d}{{\rm d}\theta}f (\theta)\end {array}
 \right].
\end{equation}

For further use, let us observe that the matrix (\ref{eq:g})
is block diagonal, so its inverse matrix is block diagonal as well,
i.e., $g^{r\phi}=g^{\theta\phi}=0$ and $g^{\theta\theta}> 0$.
%


To calculate the Cartan tensor $C_{abc}$ with $a,b,c\in\{r,\theta,\phi\}$, we can proceed analogically:
\begin{equation}\label{0002}
C(X,Y,Z)=\tfrac{1}{2}\left(Z(g(X,Y))-g(\tilde{\nabla}_ZX,Y)
-g(X,\tilde{\nabla}_ZY)\right).
\end{equation}
Using (\ref{0001}) and (\ref{eq:g}), we see that the only possibly nonzero components of the Cartan tensor are
\begin{equation} \label{eq:CT1}
\begin{array}{rcl}C_{\theta\theta\theta} &=& 2r^2
\tfrac{{\rm d}}{{\rm d}\theta}f(\theta)+\tfrac12 r^2
\tfrac{{\rm d}^3}{{\rm d}\theta^3}f(\theta),\\
C_{\theta\phi\phi}&=&-\tfrac12 r^2\cos 2\theta\tfrac{{\rm d}}{{\rm d}\theta}f(\theta)+\tfrac14 r^2\sin2\theta
\tfrac{{\rm d}^2}{{\rm d}\theta^2}f(\theta)
\end{array}\end{equation}
(of course $C_{\theta\phi\phi}= C_{\phi\theta\phi}= C_{\phi\phi\theta}$ because $C$ is symmetric). Note that it is clear in advance that  every component of the form $C_{r\cdot\cdot}=C(\tfrac{\partial}{\partial r},\cdot,\cdot)$ is zero, since $\tfrac{\partial }{\partial r}$ is the Euler vector field annihilating $C$. It is also clear by Cartan's trick that the component  $C_{\theta\theta\phi}$ is zero since the mapping given by
 $\phi\mapsto -\phi + \textrm{const}$ is in fact a linear isometry which from one side  changes the sign for $C_{\theta\theta\phi}$ and from the other side  preserves it.

 In the case $n=3$, the only curvature component we need to consider  is
\begin{equation}\label{0003}
R_{\theta\phi\phi\theta}=g(R(\tfrac{\partial}{\partial \theta},\tfrac{\partial}{\partial\phi})\tfrac{\partial}{\partial\phi},
\tfrac{\partial}{\partial\theta})
=\sum_{a,b\in\{r,\theta,\phi\}}(C_{\theta\theta a}g^{ab}
C_{\phi\phi b}-C_{\theta\phi a}g^{ab}C_{\phi\theta b}).
\end{equation}
Plugging \eqref{eq:g} and  \eqref{eq:CT1} into \eqref{0003} and using  the vanishing of $g^{\theta\phi}$, $C_{r\cdot\cdot}$ and
$C_{\theta\theta\phi}$,
we get
\begin{eqnarray*}
R_{\theta\phi\phi\theta}=C_{\theta\theta\theta}g^{\theta\theta}
C_{\theta\phi\phi}-C_{\theta\phi\phi}g^{\phi\phi}
C_{\theta\phi\phi}.
\end{eqnarray*}
So the vanishing of the Riemann curvature implies \begin{equation}\label{0010}
C_{\theta\theta\theta}g^{\theta\theta}C_{\theta\phi\phi}
-C_{\theta\phi\phi}g^{\phi\phi}C_{\theta\phi\phi}=0.\end{equation}

Note that the $\theta$-derivative of $C_{\theta\phi\phi}$ is $\tfrac{1}{2} \sin 2\theta  \ C_{\theta\theta\theta}$. Indeed,
\begin{eqnarray*}
\tfrac{{\rm d}}{{\rm d}\theta}(\tfrac{C_{\theta\phi\phi}}{r^2})
&=&\tfrac{{\rm d}}{{\rm d}\theta}\left(-\tfrac12\cos 2\theta\tfrac{{\rm d}}{{\rm d}\theta}f(\theta)+\tfrac14 \sin2\theta
\tfrac{{\rm d}^2}{{\rm d}\theta^2}f(\theta)\right)\\
&=&\tfrac12\sin 2\theta\left(2\tfrac{{\rm d}}{{\rm d}\theta}f(\theta)+\tfrac12\tfrac{{\rm d}^3}{{\rm d}\theta^3}f(\theta)
\right)=\tfrac{1}{2}\sin2\theta\cdot\tfrac{C_{\theta\theta\theta}}{r^2}.
\end{eqnarray*}
Thus, $h(\theta)=\left(\tfrac{C_{\theta\phi\phi}}{r^2}\right)^2$
is the solution of the following ODE:
\begin{equation}\label{1012}
\tfrac{{\rm d}}{{\rm d}\theta}h(\theta)=\tfrac{ g^{\phi\phi}\sin2\theta}{g^{\theta\theta}}\cdot h(\theta)=\tfrac{\cos\theta
\left(4f(\theta)^2+2f(\theta)\tfrac{{\rm d}^2}{{\rm d}t^2}f(\theta)-(\tfrac{{\rm d}}{{\rm d}t}f(\theta))^2\right)}{f(\theta)\left(2\sin\theta f(\theta)+\cos\theta
\tfrac{{\rm d}}{{\rm d}t}f(\theta)\right)}\cdot h(\theta)
\end{equation}
on $(0,\pi)$.

From (\ref{eq:g}) we see that $g_{\phi\phi}$ at the points of $S_F$, i.e.,
when $r^2=\tfrac{1}{2f(\theta)}$, is given by
\begin{eqnarray*}
\sin^2\theta+\tfrac{\sin\theta\cos\theta}{2f(\theta)}
\tfrac{{\rm d}}{{\rm d}\theta}f(\theta).
\end{eqnarray*}
In particular, we have $g_{\phi\phi}=1$ at $x\in S_F$ with $\theta$-coordinate equal to $\pi/2$. So the $g$-arc length of the curve
 $\{\theta=\pi/2\}$ on $S_F$ is $2\pi$. When we identify $(S_F,g)$
with a standard $S^2(1)$, the $SO(2)$-action which shifts the
$\phi$-coordinates on $S_F$ coincides with a standard linear $SO(2)$-action on $S^2(1)$ which orbits are the latitude lines.
The curve $\theta=\pi/2$ on $S_F$ corresponds to the equator which has
the maximal length among all latitude lines. So we have
\begin{eqnarray*}
\tfrac{{\rm d}}{{\rm d}\theta}\left(g_{\phi\phi}{}_{|r=(2f(\theta))^{-1/2}}
\right)_{|\theta=\pi/2}
&=&\tfrac{{\rm d}}{{\rm d}\theta}\left(\sin^2+\tfrac{\sin\theta\cos\theta}{2f(\theta)}
\tfrac{{\rm d}}{{\rm d}\theta}f(\theta)\right)_{|\theta=\pi/2}\\
&=&
-\tfrac{1}{2f(\pi/2)}\tfrac{{\rm d}}{{\rm d}\theta}f(\theta)_{|\theta=\pi/2}=0,
\end{eqnarray*}
i.e., $\tfrac{{\rm d}}{{\rm d}\theta}f(\theta)_{|\theta=\pi/2}=0$.
Plugging  it into the formula of $C_{\theta\phi\phi}$ in (\ref{eq:CT1}),
we see $C_{\theta\phi\phi}=0$ when $\theta=\pi/2$ and then
$h(\pi/2)=0$. Thus,  $h(\theta)$  satisfies the ODE   (\ref{1012}) with the initial condition $h(\pi/2)=0$ so it is identically zero.  Hence the Cartan tensor of  $F$ vanishes identically which implies that the third partial derivatives of $F^2$  with respect to linear coordinates vanish   so $F$ is Euclidean. Lemma \ref{lem:final}  is proved for $n=3$.


\label{subsection-2-2}
\begin{remark} \label{Rem:0040}
The equality (\ref{0010}) follows from
(and is fact is equivalent to)
\begin{equation}\label{1011}
C_{\theta\theta\theta}g^{\theta\theta}-C_{\theta\phi\phi}g^{\phi\phi}=0
\end{equation}
everywhere on $\mathbb{R}^3\backslash\{0\}$. This is a 3rd order ODE for $f(\theta)$, and has a 3-parameter family of local solutions. Among these local solutions,  $c_1+c_2\cos 2\theta$ with appropriate constants $c_1$ and $c_2$ corresponds to the Euclidean norms. So we may generically perturb it among local solutions of (\ref{1011}), and use the resulting $f(\theta)$ to construct a flat Hessian metric $g={\rm d}^2 E$ for $E=r^2 f(\theta)$ in some
conic open subset of $\mathbb{R}^3\backslash\{0\}$. Local Hessian isometries can be constructed between $g$ and the Hessian metric for an Euclidean norm. These local Hessian isometries are not linear.
%
%
\end{remark}

Let us now  prove Lemma \ref{lem:final} when $n>3$. Let $y_0\neq0$ be any point fixed by the action of $O(n-1)$, and $V_0$
any 3-dimensional vector subspace containing $y_0$. We can find
an involution in $O(n-1)$, such that $V_0$ is its fixed point set. Indeed, we can find suitable orthonormal coordinates $(x_1,\cdots,x_n)$ on $\mathbb{R}^n$, such that
$V_0$ consists of all vectors $(x_1,x_2,x_3,0,\cdots,0)$ and $y_0$ is
presented by $(a_0,0,\cdots,0)$. Then $V_0$ is the fixed point set of
the mapping $(x_1,\cdots,x_n)\mapsto(x_1,x_2,x_3,-x_4,\cdots,-x_n)$
in $O(n-1)$.

The restriction $F_0=F_{|V_0}$ is invariant with respect to the standard block diagonal action of $O(2)=SO(1)\times O(2)$. Its Hessian metric
$g_0=\tfrac12{\rm d}^2F_0^2=g_{|V_0\backslash\{0\}}$ is flat because
it is the restriction of the  ambient metric $g$ which is   flat to a automatically  totally geodesic fixed points set.
Then, $F_0$ is Euclidean.
By the $O(n-1)$-invariancy of $F$, we see  that $F$ is
Euclidean as well.
\end{proof}

\subsection{Proof of Theorem \ref{thm:main1} for
 $2\leq k \leq n/2$. }
\label{subsection-2-3}

Assume now the Minkowski norm $F$  on $\mathbb{R}^n$
is invariant with respect to the standard block diagonal action on $O(k)\times O(n-k)$ with $2\leq k\leq n/2$.
We denote by $G_0$ the connected isometry group for $(\mathbb{R}^n\backslash\{0\},g)$ and for $(S_F,g_{|S_F})$.

We first  consider
the case when $(S_F, g_{|S_F})$ is a homogeneous Riemannian sphere.
As in the previous section, let us  apply Cartan's trick to prove that the
 Cartan tensor vanishes at the point
$
y_0= (a_0,0,...,0)\in S_F
$.
Let $v\in\mathbb{R}^n$ be any vector
contained in the tangent space $T_{y_0}S_F$, then its $x_1$-coordinate vanishes.
The linear isometry $
(x_1,...,x_n)\mapsto (x_1,-x_2,-x_3. ..., -x_n)
$
 in $O(k)\times O(n-k)$ fixes $y_0$ and its tangent map at $y_0$ sends $v\in T_{y_0}S_F$ to $-v$. It preserves the Cartan tensor, so we have
$$C(v, v, v)= C(-v, -v, -v)=-C(v,v,v)$$ at $y_0$ for each $v\in T_{y_0}S_F$, which implies $C=0$ there.

Using \eqref{eq:fact} and the same argument as for Lemma \ref{Lem:2}, we see $(S_F,g_{|S_F})$ has constant curvature 1 and
$(\mathbb{R}^n\backslash\{0\},g)$ is flat.
By $2\leq k\leq n/2$, we have $n\geq 4$, and the absolute 1-homogeneity for the $SO(k)\times SO(n-k)$-invariant Minkowski norm $F$. By Theorem
\ref{thm:Brickell}, we obtain that $F$ is an Euclidean norm, which ends the proof of Theorem \ref{thm:main1} when $(S_F,g_{|S_F})$ is a homogeneous Riemannian sphere.

Next, we consider the case when $(S_F,g_{|S_F})$ is not a homogeneous
Riemannian sphere. Since the $SO(k)\times SO(n-k)$-action on $S_F$
has  cohomogeneity one, $G_0$ must preserve each $SO(k)\times SO(n-k)$-orbit. Then the $G_0$-action maps normal geodesics on $(S_F,g_{|S_F})$ (i.e., geodesics on $(S_F,g_{|S_F})$ which are orthogonal to all the $SO(k)\times SO(n-k)$-orbits) to normal geodesics on $(S_F,g_{|S_F})$.
So each $\Phi\in G_0$ is determined by its restriction to any
principal orbit $\mathcal{O}=(SO(k)\times SO(n-k))\cdot x$, which results in an injective Lie group homomorphism from $G_0$ to the
isometry group for $(\mathcal{O},g_{|\mathcal{O}})$.

The restriction of the Hessian metric $g$ to
the principal orbit
\begin{eqnarray*}
\mathcal{O}&=&(SO(k)\times SO(n-k))\cdot x
=(SO(k)\times SO(n-k))/(SO(k-1)\times SO(n-k-1))\\
&=&
(SO(k)/SO(k-1))\times (SO(n-k)/SO(n-k-1))
\end{eqnarray*}
is isometric to the Riemannian product of two standard spheres,
with dimensions $k-1$ and $n-k-1$ respectively. The isometry group
for $(\mathcal{O},g)$ has the Lie algebra $so(k)\oplus so(n-k)$, so we have $\dim G_0\leq
\dim SO(k)\times SO(n-k)$. On the other hand $G_0$ contains all the linear $SO(k)\times SO(n-k)$-actions. Thus, we have $G_0=SO(k)\times SO(n-k)$ also in this case.
Theorem \ref{thm:main1} is proved.

 \section{Local Hessian isometry which maps  orbits to orbits}
 \label{section-3}

\subsection{Spherical coordinates presentation for local Hessian isometries}
\label{subsection-3-1}
Assume the integers $k$ and $n$ satisfy $n\geq 3$ and $1\leq k\leq n/2$.

The subgroup $O(k)\times O(n-k)$ of  $O(n)$, consisting of
$\mathrm{diag}(A,B)$ for all $A\in O(k)$ and $B\in O(n-k)$, has the
standard block diagonal action on the Euclidean $\mathbb{R}^n$
of column vectors, with respect to which we have the orthogonal linear decomposition $\mathbb{R}^n=V'\oplus V''$, where $V'$ and $V''$ are $k$- and $(n-k)$-dimensional $O(k)\times O(n-k)$-invariant subspaces respectively. For simplicity, if not otherwise specified, orbits are referred to $O(n-1)$-orbits (which are the same as $SO(1)\times O(n-1)$- and $SO(n-1)$-orbits) when $k=1$,
and $O(k)\times O(n-k)$-orbits (which are the same as $SO(k)\times SO(n-k)$- and $SO(k)\times O(n-k)$-orbits) when $k>1$.

With the marking point $y\in\mathbb{R}^n\backslash\{0\}$ fixed,
the orthonormal coordinates
$(x_1,\cdots,x_n)^T$ can and will be chosen such that
\begin{enumerate}
\item
$V'$ and $V''$ are represented by
$x_{k+1}=\cdots=x_n=0$ and $x_1=\cdots=x_k=0$ respectively;
\item
The marking point $y$ has coordinates $(y_1,0,\cdots,0,y_{k+1},0,\cdots,0)^T$ with $y_1\geq 0$ and $y_{k+1}\geq 0$.
\end{enumerate}
Denote by
\begin{eqnarray*}
S'&=&\{(x_1,\cdots,x_k)^T| x_1^2+\cdots+x_k^2=1\}\quad\mbox{and}\\
S''&=&\{(x_{k+1},\cdots,x_n)^T| x_{k+1}^2+\cdots+x_n^2=1\}
\end{eqnarray*}
the $(k-1)$- and $(n-k-1)$-dimensional standard unit spheres respectively.  Then we set the spherical coordinates as following.

If $k=1$, the spherical coordinates $(r,\theta,\xi)\in\mathbb{R}_{>0}\times
(0,\pi)\times S''$ are determined by
$$x_1=r\cos\theta\quad\mbox{and}\quad
(x_2,\cdots,x_n)^T=r\sin\theta\cdot \xi,$$
which are well defined on $\mathbb{R}^n\backslash V'$.
The action of $A\in O(n-1)$ (i.e., $\mathrm{diag}(1,A)\in SO(1)\times O(n-1)\subset O(n)$) fixes $r$ and $\theta$ and changes $\xi$ to $A\xi$.

If $k>1$, the spherical coordinates $(r,\theta,\xi',\xi'')\in
\mathbb{R}_{>0}\times (0,\pi/2)\times S'\times S''$ are determined by
$$(x_1,\cdots,x_k)^T=r\cos\theta\cdot \xi'
\quad\mbox{and}\quad
(x_{k+1},\cdots,x_n)^T=r\sin\theta\cdot \xi'',$$
which are well defined on $\mathbb{R}^n\backslash (V'\cup V'')$.
The action of $\mathrm{diag}(A',A'')\in O(k)\times O( n-k)$
fixes $r$ and $\theta$, and changes $\xi'$ and $\xi''$ to $A'\xi'$ and $A''\xi''$
respectively.

Let us now  consider two
$SO(k)\times SO(n-k)$-invariant  Minkowski norms
$F_1$ and $F_2$ on $\mathbb{R}^n$, and denote their
Hessian metrics by $g_1=g_1(\cdot,\cdot)$ and $g_1=g_2(\cdot,\cdot)$
respectively. To distinguish the different norms or Hessian metrics, we use
$t$ to denote the $\theta$-coordinate where $F_1$ or $g_1$ is concerned, but still call it the $\theta$-coordinate.
By the homogeneity and
$SO(k)\times O(n-k)$-invariancy,
$E_i=\tfrac12F_i^2$ can be presented by
spherical coordinates as
\begin{equation*}
E_1=r^2 f(t)\quad \mbox{and}\quad E_2=r^2 h(\theta)
\end{equation*}
respectively. Though $t$ and $\theta$ belongs to $(0,\pi)$ or $(0,\pi/2)$, $f(t)$ and $h(\theta)$  can be periodically extended
to  even positive smooth functions on $\mathbb{R}$, with the period $2\pi$ or $\pi$, when $k=1$ or $k>1$ respectively.

Without loss of generality, we will further assume $y\in S_{F_1}$.
The $SO(k)\times O(n-k)$-action on $(S_{F_i},g_{i})$
is of cohomogeneity one. The normal geodesics on $(S_{F_i},g_i)$ are those which intersect orbits orthogonally. Using
fixed point set technique and similar Cartan's trick as in the proof of Lemma \ref{Lem:2}, it is easy to see that around
any principal orbit, normal geodesics are characterized by the following equations for spherical coordinates, $\xi\equiv\mathrm{const}$ when $k=1$, or $(\xi',\xi'')\equiv\mathrm{const}$ when $k>1$.

Now we assume $y$ satisfies (\ref{eq:notzero}) in Theorem \ref{thm:mainA}, i.e.,
\begin{equation*}
g_1(v',v'')\neq0\mbox{ at }y,\quad\mbox{for some }v'\in V'\mbox{ and }v''\in V'',
\end{equation*}
Applying Cartan's trick to those
$\mathrm{diag}(\pm1,\cdots,\pm 1)\in O(k)\times O(n-k)$
which preserves $F_1$ and fix $y$,
we see easily
\begin{enumerate}
\item
When $k=1$, we have $y\notin V'$, and when $k>1$, $y\notin V'\cup V''$. So the spherical coordinates of $y$ are well defined.
\item
The Hessian matrix $(a_{ij})=(g_1(\tfrac{\partial}{\partial x_i},
\tfrac{\partial}{\partial x_j}))$ is blocked-diagonal. To be precise,
we have at $y$
\begin{equation}\label{017}
g_1(\tfrac{\partial}{\partial x_i},\tfrac{\partial}{\partial x_j})=0,\quad \mbox{when }i\neq j\mbox{ and } \{i,j\}\neq\{1,k+1\}.
\end{equation}
\end{enumerate}

Using the spherical coordinates, the assumption (\ref{eq:notzero}) can be interpreted as following.

\begin{lemma} \label{lemma-002}
The following statements are equivalent (no matter $k=1$ or $k>1$):
\begin{enumerate}
\item
The marking point $y\in\mathbb{R}^n\backslash\{0\}$ satisfies (\ref{eq:notzero});
\item We have $g_1(\tfrac{\partial}{\partial x_1},\tfrac{\partial}{\partial x_{k+1}})\neq0$ at $y$;
\item The $\theta$-coordinate of $y$ satisfies
\begin{equation}\label{005}
-\cos t\sin t\tfrac{{\rm d}^2}{{\rm d}t^2}f(t)+(\cos^2 t-\sin^2 t)
\tfrac{{\rm d}}{{\rm d}t}f(t)\neq0.
\end{equation}
\end{enumerate}
Furthermore, $F_1$ is not locally Euclidean around $y$ when $y$ satisfies (\ref{eq:notzero}).
\end{lemma}
\begin{proof}
Because of (\ref{017}) at $y$, (1) and (2) in Lemma \ref{lemma-002} are equivalent. Further discussion can be reduced to the 3-dimensional subspace $V$ given by
$x_2=\cdots=x_k=x_{k+2}=\cdots=x_{n-1}=0$. By similar calculation as for (\ref{eq:g}), we get
\begin{eqnarray*}
g_1(\tfrac{\partial}{\partial x_1},\tfrac{\partial}{\partial x_{k+1}})&=&\pm g_1(\cos t\tfrac{\partial}{\partial r}-\tfrac{1}{r}\sin t
\tfrac{\partial}{\partial t},\sin t\tfrac{\partial}{\partial r}+
\tfrac{\cos t}{r}\tfrac{\partial}{\partial t})\\
&=&\pm\left(-\cos t\sin t\tfrac{{\rm d}^2}{{\rm d}t^2}f(t)+(\cos^2t-\sin^2t)
\tfrac{{\rm d}}{{\rm d}t}f(t)\right).
\end{eqnarray*}
Then the equivalence between (2) and (3) in Lemma \ref{lemma-002} follows immediately. Finally, we compare (\ref{005}) with the formula for $C_{\theta\phi\phi}$ in (\ref{eq:CT1}), we see the Cartan tensor does not vanish at $y$ when (\ref{eq:notzero})
is satisfied. So $F_1$ is not locally Euclidean there.
\end{proof}

Let us   consider a local Hessian isometry $\Phi$ from $F_1$ to $F_2$ which is defined on an $SO(k)\times O(n-k)$-invariant
conic neighborhood of $y$, and maps orbits to orbits. Notice that
$\Phi$ satisfies the positive 1-homogeneity and preserves the norm.

The following spherical coordinates presentations of $\Phi$ are crucial for proving Theorem \ref{thm:mainA} and Theorem \ref{thm:mainB}.

\begin{lemma}\label{spherical-presenting-hessian-isometry-1}
When $k=1$, the local Hessian isometry $\Phi$ can be presented by
spherical coordinates as
\begin{equation}\label{presenting-1}
(r,t,\xi)\mapsto (\tfrac{f(t)^{1/2}}{h(\theta(t))^{1/2}}\cdot r,\theta(t),A\xi)
\end{equation}
in some $O(n-1)$-invariant conic neighborhood of $y$, where
$A\in O(n-1)$ and $\theta(t)$ is a smooth function with nonzero derivatives everywhere.
\end{lemma}

\begin{proof}
By the homogeneity of $\Phi$, to prove (\ref{presenting-1}), we only need to discuss $\Phi(x)$ for $x\in S_{F_1}$. When $x$ is sufficiently close to $y\in S_{F_1}$, $x\notin V'$, so its spherical coordinates
$(r,t,\xi)=((2f(t))^{-1/2},t,\xi)$ are well defined. Since $\Phi$ maps principal orbits on $S_{F_1}$ to
principal orbits on $S_{F_2}$, and  each principal orbit
is characterized by constant $\theta$-coordinates, we see that the $\theta$-coordinate of
$\Phi(x)$ only depends on $t$. So we may
denote it as $\theta(t)$, which smoothness is obvious.
Since $F_1(x)=F_2(\Phi(x))=1$, the $r$-coordinate of $\Phi(x)$
is $(2h(\theta(t)))^{-1/2}=\tfrac{f(t)^{1/2}}{h(\theta(t))^{1/2}}\cdot r$.

Denote $\mathcal{O}_1=O(n-1)\cdot x$ the principal orbit in
$S_{F_1}$ passing $x$. When endowed
with the Hessian metric, it is a homogeneous Riemannian sphere
$O(n-1)/O(n-2)$, which is isometric to a radius $R$ standard sphere (i.e., its perimeter is $2\pi R$ when $n=3$ or it has constant curvature $R^{-1}$ when $n>3$). For $\mathcal{O}_2=O(n-1)\cdot\Phi(x)$ in $S_{F_2}$, we have a similar
claim. Since
the local Hessian isometry $\Phi$ maps $\mathcal{O}_1$
onto $\mathcal{O}_2$, $(\mathcal{O}_2,g_2)$ is also isometric to a radius $R$ standard sphere.
Denote $g_{\mathrm{st}}$ the standard unit sphere metric on $S''$,
then the $O(n-1)$-equivariant diffeomorphism $\Phi_1:(\mathcal{O}_1,g_1)\rightarrow
(S'',R^2 g_{\mathrm{st}})$,
mapping $x'\in \mathcal{O}_1$ to its
$\xi$-coordinate, is an isometry.
Similarly, we have another homothetic correspondence $\Phi_2:(\mathcal{O}_2,g_2)\rightarrow(S'',R^2g_{\mathrm{st}})$.
The composition
$$\Psi=\Phi_2\circ\Phi\circ\Phi_1^{-1}:
(S'',R^2g_{\mathrm{st}})\rightarrow
(S'',R^2g_{\mathrm{st}}),$$
which characterizes how the local Hessian isometry $\Phi$ changes the $\xi$-coordinates, is an isometry. So $\Psi$ must be of the form
$\xi\mapsto A\xi$ for some $A\in O(n-1)$.

Since $\Phi$ maps orbits on $S_{F_1}$
to orbits on $S_{F_2}$, it also maps normal geodesics to normal geodesics. Normal geodesics have
constant $\xi$-coordinates around each principal orbit. So the matrix $A\in O(n-1)$ in the presentation of $\Psi$
does not depend on $t$.

 Above argument proves the spherical coordinates presentation of $\Phi$ in (\ref{presenting-1}). Then we prove $\theta(t)$ has nonzero derivatives everywhere.

 We use (\ref{presenting-1}) to calculate the tangent map $\Phi_*$ at $x$, which can be presented as
the following Jacobi matrix
$$\left(
    \begin{array}{ccc}
      \tfrac{f(t)^{1/2}}{h(\theta(t))^{1/2}} &
      \tfrac{h(\theta(t))-f(t)\tfrac{{\rm d}}{{\rm d}t}\theta(t)}{2f(t)(2h(\theta(t)))^{3/2}} & 0 \\
      0 & \tfrac{{\rm d}}{{\rm d}t}\theta(t) & 0 \\
      0 & 0 & A \\
    \end{array}
  \right).
$$
Since $\Phi$ is a local diffeomorphism, its Jacobi matrix must have
nonzero determinant, which requires $\tfrac{{\rm d}}{{\rm d}t}\theta(t)\neq0$.
\end{proof}


\begin{lemma}\label{spherical-presenting-hessian-isometry-2}
When $k>1$, the local Hessian isometry $\Phi$ can be presented by
spherical coordinates either as
\begin{equation}\label{presenting-2}
(r,t,\xi',\xi'')\mapsto (\tfrac{f(t)^{1/2}}{h(\theta(t))^{1/2}}\cdot r,\theta(t),A'\xi',A''\xi'')
\end{equation}
or as
\begin{equation}\label{presenting-3}
(r,t,\xi',\xi'')\mapsto (\tfrac{f(t)^{1/2}}{h(\theta(t))^{1/2}}\cdot r,\theta(t),A''\xi'', A'\xi')
\end{equation}
in some $O(k)\times O(n-k)$-invariant conic neighborhood of $y$, where
$A'\in O(k)$, $A''\in O(n-k)$, $\theta(t)$ is a smooth function with nonzero derivatives everywhere, and (\ref{presenting-3}) may happen only when $n=2k$.
\end{lemma}

\begin{proof} We only need to discuss the spherical coordinates of
$\Phi(x)$ for $x\in S_{F_1}$ sufficiently close to $y$.
Denote the orbits
\begin{eqnarray*}
& &\mathcal{O}_1=(O(k)\times O(n-k))\cdot x,\
\mathcal{O}'_1=(O(k)\times\{e\})\cdot x,\
\mathcal{O}''_1=(\{e\}\times O(n-k))\cdot x,\\
& &\mathcal{O}_2=(O(k)\times O(n-k))\cdot\Phi(x),\
\mathcal{O}'_2=(O(k)\times\{e\})\cdot \Phi(x),\
\mathcal{O}''_2=(\{e\}\times O(n-k))\cdot \Phi(x).
\end{eqnarray*}
When endowed with the restriction of $g_1$,
$\mathcal{O}_1=(O(k)\times O(n-k))/(O(k-1)\times O(n-k-1))$ is
the Riemannian product of the two homogeneous Riemannian spheres,
i.e., $\mathcal{O}'_1=O(k)/O(k-1)$, which is isometric to a radius $R'_1$ standard sphere, and $\mathcal{O}''_1=O(n-k)/O(n-k-1)$, which is isometric to a radius $R''_1$ standard sphere.
Denote $g'_{\mathrm{st}}$ and $g''_{\mathrm{st}}$ the standard unit sphere metrics on
$S'$ and $S''$ respectively, and $g_{R'_1,R''_1}$ the product metric of ${R'_1}^2g'_{\mathrm{st}}$ and
${R''_1}^2g''_{\mathrm{st}}$ on $S'\times S''$. Then the $O(k)\times O(n-k)$-equivariant diffeomorphism
$\Phi_1:(\mathcal{O}_1,g_1)\rightarrow
(S'\times S'',g_{R'_1,R''_1})$ is an isometry.
Similarly, $(\mathcal{O}'_2,g_2)$ and $(\mathcal{O}''_2,g_2)$ are isometric
to standard spheres with radii $R'_2$ and $R''_2$ respectively, and
we have another isometry
$$\Phi_2:(\mathcal{O}_2,g_2)\rightarrow
(S'\times S'',g_{R'_2,R''_2}).$$ Since the local Hessian isometry $\Phi$ maps
$\mathcal{O}_1$ onto $\mathcal{O}_2$,
the composition
$$\Psi=\Phi_2\circ\Phi\circ\Phi_1^{-1}:(S'\times S'',g_{R'_1,R''_1})\rightarrow (S'\times S'',g_{R'_2,R''_2}),$$
which characterizes how the local isometry $\Phi$ changes the $\xi'$- and $\xi''$-coordinates, is an isometry. The isometries on the Riemannian product of two standard spheres are completely known. There are two possibilities:
\begin{enumerate}
\item $\Psi(\xi',\xi'')=(A'\xi',A''\xi'')$ for some $A'\in O(k)$ and $A''\in O(n-k)$, $R'_1=R'_2$ and $R'_2=R'_1$.
\item $n=2k$, $\Psi(\xi',\xi'')=(A''\xi'',A'\xi')$ for some $A',A''\in O(k)$,
$R'_1=R''_2$ and $R'_2=R''_1$.
\end{enumerate}
For each possibility, $\Psi$ represents a distinct homotopy class, which
does not change when we move $x$ continuously. Further more,
$A'$ and $A''$ in the presentation of $\Psi$ are independent of $t$, because $\Phi$ maps normal geodesics on $S_{F_1}$ to those on $S_{F_2}$, and normal geodesics on $S_{F_i}$ have constant $\xi'$- and $\xi''$-coordinates.

The remaining  arguments are similar to those for Lemma \ref{spherical-presenting-hessian-isometry-1}, so we skip them.\end{proof}

\subsection{Equivariant Hessian isometries}

Analyse the spherical coordinates presentations (\ref{presenting-1}), (\ref{presenting-2}) and (\ref{presenting-3}) in Lemma \ref{spherical-presenting-hessian-isometry-1} and Lemma \ref{spherical-presenting-hessian-isometry-2}, we see immediately that a local Hessian isometry $\Phi$ mapping orbits to orbits can be decomposed as $\Phi=\Phi_1\circ\Phi_2$, in which $\Phi_1$ is a linear isometry mapping orbits to orbits, and $\Phi_2$ is a local Hessian isometry fixing all $\xi$-coordinates when $k=1$, or fixing all $\xi'$- and $\xi''$-coordinates when $k>1$. For example, when
$n=2k$ and $\Phi$ is presented by spherical coordinates as in (\ref{presenting-3}),
i.e. $(r,t,\xi',\xi'')\mapsto (\tfrac{f(t)^{1/2}r}{h(\theta(t))^{1/2}},A''\xi'',A'\xi')$, $\Phi_1$
is the action of $\left(
                 \begin{array}{cc}
                   0 & A'' \\
                   A' & 0 \\
                 \end{array}
               \right)
 $ in $O(n)$. It maps orbits to orbits, exchanging the curvature constants of the two product factors in the orbit, and it induces a new $O(k)\times O(n-k)$ norm $F_3=F_2\circ\Phi_1$.
The composition $\Phi_2=\Phi_1^{-1}\circ\Phi$ is local Hessian isometry between
from $F_1$ to $F_3$ fixing $\xi'$- and $\xi''$-coordinates.

For simplicity, we call $\Phi$ {\it equivariant}
if it equivariant with respect to the $O(n-1)$-action or the $O(k)\times
O(n-k)$-action when $k=1$ or $k>1$ respectively. Practically, we will only use those equivariant $\Phi$ which
fix all $\xi$-coordinates or all $\xi'$- and $\xi''$-coordinates.


Summarizing above observations, we have the following theorem.

\begin{theorem}\label{thm-decomposition}
Any local Hessian isometry $\Phi$ between two $SO(k)\times SO(n-k)$-invariant Minkowski norms with $n\geq 3$ and $1\leq k\leq n/2$ which maps orbits to orbits can be decomposed as $\Phi=\Phi_1\circ\Phi_2$, in which $\Phi_1$ is a linear isometry and $\Phi_2$ is an equivariant local Hessian isometry fixing all the $\xi$-coordinates or all the
$\xi_1$- and $\xi_2$-coordinates.
\end{theorem}

The following examples of global equivariant Hessian isometries are crucial
for the proofs of Theorem \ref{thm:mainA}.

\begin{example}\label{example-linear}
Let $F_1$ be any $SO(k)\times O(n-k)$-invariant
Minkowski norm on $\mathbb{R}^n$, and  $\Phi$ a linear map
\begin{equation}
(x_1,\cdots,x_k,x_{k+1},\cdots,x_n)\mapsto(ax_1,\cdots,ax_k,bx_{k+1},
\cdots,bx_{n}),
\end{equation}
with the parameter pair $(a,b)\in\mathbb{R}_{\neq0}\times\mathbb{R}_{>0}$ when $k=1$, or $(a,b)\in\mathbb{R}_{>0}\times\mathbb{R}_{>0}$ when $k>1$. Then $\Phi$
induces another $SO(k)\times O(n-k)$-invariant Minkowski norm $F_2=F_1\circ\Phi^{-1}$, such that
$\Phi$ is an equivariant Hessian isometry from $F_1$ to $F_2$ which fixes all $\xi$- or all $\xi'$- and $\xi''$-coordinates. We will simply call it the linear example with the parameter pair $(a,b)$.
\end{example}

If $k=1$, the function
$\theta(t)$ in the spherical coordinates presentation for the linear example with the parameter pair $(a,b)$ is
\begin{equation}
\label{theta-linear}
\theta(t)=\arccos\left(\tfrac{a\cos t}{(a^2\cos^2 t+b^2\sin^2 t)^{1/2}}\right),
 \end{equation}
for $t\in(0,\pi)$. It satisfies
\begin{equation}\label{ODE-1}
\tfrac{{\rm d}}{{\rm d}t}\theta(t)=\tfrac{\sin\theta(t)\cos\theta(t)}{\sin t\cos t},
\end{equation}
when $t\neq\pi/2$.

If $k>1$,  the function
$\theta(t)$ satisfies (\ref{theta-linear}) and (\ref{ODE-1})
for $t\in(0,\pi/2)$.

\begin{example}\label{example-Legendre}
Let $F_1$ be any $SO(k)\times O(n-k)$-invariant  Minkowski norm
on $\mathbb{R}^n$, and $\Phi:\mathbb{R}^n\backslash\{0\}\rightarrow
\mathbb{R}^n\backslash\{0\}$ the diffeomorphism
\begin{eqnarray}
(x_1,\cdots,x_k,x_{k+1},\cdots,x_n)
\mapsto (a\tfrac{\partial E_1}{\partial x_1},\cdots,a\tfrac{\partial E_1}{\partial x_k},
b\tfrac{\partial E_1}{\partial x_{k+1}}, \cdots,b\tfrac{\partial E_1}{\partial x_n}),
\end{eqnarray}where $E_1=\tfrac12F_1^2=r^2f(t)$ in spherical coordinates, and the requirement for the parameter pair $(a,b)$ is the same as in Example \ref{example-linear}. Then $\Phi$
induces another  $SO(k)\times O(n-k)$-invariant Minkowski norm $F_2=F_1\circ\Phi^{-1}$, such that
$\Phi$ is an equivariant Hessian isometry from $F_1$ to $F_2$ which
fixes $\xi$- or all $\xi'$- and $\xi''$-coordinates.
We will simply call it the Legendre example with the parameter pair $(a,b)$, because it is the composition between the Legendre
transformation of $F_1$, from $F_1$ to $\hat{F}_1$,
and a linear isometry from $\hat{F}_1$ to $F_2$.
\end{example}

If $k=1$,
the function $\theta(t)$ in the spherical coordinates presentation for the Legendre example with the parameter pair $(a,b)$ is
\begin{equation}
\label{theta-legendre}
\theta(t)=
\arccos\left(\tfrac{2\cos t f(t)-\sin t \tfrac{{\rm d}}{{\rm d}t}f(t)}{\left[4(\cos^2 t+\tfrac{b^2}{a^2}\sin^2 t)f(t)^2
+4(\tfrac{b^2}{a^2}-1)\cos t\sin t f(t)\tfrac{{\rm d}}{{\rm d}t}f(t)+
(\sin^2 t+\tfrac{b^2}{a^2}\cos^2 t)\left(\tfrac{{\rm d}}{{\rm d}t}f(t)\right)^2\right]^{1/2}}\right),
\end{equation}
for $t\in(0,\pi)$. It
satisfies
\begin{equation}\label{ODE-2}
\tfrac{{\rm d}}{{\rm d}t}\theta(t)=\tfrac{
\left(2f(t)\tfrac{{\rm d}^2}{{\rm d}t^2}f(t)-\left(\tfrac{{\rm d}}{{\rm d}t}f(t)\right)^2+4f(t)^2\right)\sin\theta(t)\cos\theta(t)}{
\left(\cos t \tfrac{{\rm d}}{{\rm d}t}f(t)+2\sin t f(t)\right)\left(-\sin t
\tfrac{{\rm d}}{{\rm d}t}f(t)+2\cos t f(t)\right)},
\end{equation}
when $\cos t \tfrac{{\rm d}}{{\rm d}t}f(t)+2\sin t f(t)\neq0$ and
$-\sin t\tfrac{{\rm d}}{{\rm d}t}f(t)+2\cos t f(t)\neq0$.

By the strong convexity of $F_1$, non-vanishing of $\cos t \tfrac{{\rm d}}{{\rm d}t}f(t)+2\sin t f(t)\neq0$ is always guaranteed for $t\in(0,\pi)$. In particular, when $n=3$, $\cos t \tfrac{{\rm d}}{{\rm d}t}f(t)+2\sin t f(t)\neq0$ is a product factor in $g_{\phi\phi}$ (see (\ref{eq:g})). Meanwhile, by the calculation
$$\tfrac{\partial E_1}{\partial x_1}=r\left(-\sin t
\tfrac{{\rm d}}{{\rm d}t}f(t)+2\cos t f(t)\right),$$
we see that $-\sin t
\tfrac{{\rm d}}{{\rm d}t}f(t)+2\cos t f(t)$ vanishes iff
$\tfrac{\partial E_1}{\partial x_1}=0$. By the strong convexity and $O(n-1)$-invariancy of $F_1$, the equation $-\sin t
\tfrac{{\rm d}}{{\rm d}t}f(t)+2\cos t f(t)=0$ has a unique solution $t'$
in $(0,\pi)$. In particular, when $f(t)\equiv f(\pi-t)$ for $t\in(0,\pi)$, $t'=\pi/2$.

If $k>1$, the corresponding function $\theta(t)$ satisfies (\ref{theta-legendre}) and (\ref{ODE-2}) for all $t\in(0,\pi/2)$.

\label{subsection-3-2}

\subsection{Proof of Theorem \ref{thm:mainA}: reduction to $n=3$}
\label{subsection-3-3}

In the following two subsections, we prove Theorem \ref{thm:mainA}. In this subsection,
we explain why and how we can reduce the proof to the case $n=3$. Then in the next subsection, we prove Theorem
\ref{thm:mainA} when $n=3$.

Let $C(U_1)$ be any $SO(k)\times SO(n-k)$-invariant connected conic open subset in $\mathbb{R}^n$ which satisfies (\ref{eq:notzero}) everywhere, and $\Phi$ a local Hessian isometry from $F_1$ to $F_2$ which is defined on $C(U_1)$ and maps orbits to orbits. By Theorem \ref{thm-decomposition}, we only need to prove Theorem \ref{thm:mainA} assuming that $\Phi$ fixes all $\xi$- or all $\xi'$- and $\xi''$-coordinates.
Then $\Phi$ preserves the 3-dimensional subspace $V$ given by $x_2=\cdots=x_k=x_{k+2}=\cdots=x_{n-1}=0$. Furthermore, when $k>1$, $\Phi$ preserves the subset $V_{x_1>0}\subset V$ with positive $x_1$-coordinates.
The restrictions ${F_i}_{|V}$ are Minkowski norms on $V$ which are invariant with respect to the subgroup $O(2)\subset O(n-1)$ fixing each point of $V^\perp$ given by $x_1=x_{k+1}=x_n=0$.
Denote $C(U'_1)$ the following $SO(2)$-invariant
connected conic open subset of $V$. When $k=1$, $C(U'_1)=C(U)\cap V$, and when $k>1$, $C(U'_1)=C(U_1)\cap V_{x_1>0}$. The restrictions ${g_i}_{|V}$ coincide with the Hessian metrics for ${F_i}_{|V}$, so
the restriction $\Phi_{|C(U'_1)}$ is a local Hessian isometry from ${F_1}_{|V}$ to ${F_2}_{|V}$. Each $SO(2)$-orbit in $C(U'_1)$ is the
intersection of an $SO(k)\times O(n-k)$-orbit with $V$ or $V_{x_1>0}$. So $\Phi_{|C(U'_1)}$ maps $O(2)$-orbits to $O(2)$-orbits.
By (\ref{017}) and Lemma \ref{lemma-002}, we have
$$g_1(\tfrac{\partial}{\partial x_{k+1}},\tfrac{\partial}{\partial x_{n}})=g_1(\tfrac{\partial}{\partial x_1},\tfrac{\partial}{\partial x_n})=0\mbox{ and }
g_1(\tfrac{\partial}{\partial x_1},\tfrac{\partial}{\partial x_k})\neq0$$
at any $y=(y_1,0,\cdots,0,y_{k+1},0,\cdots,0)\in C(U'_1)$.
So to summarize, we have\medskip

\noindent{\bf Observation 1:}\
$C(U'_1)$ and $\Phi_{|C(U'_1)}$  meet the requirements in Theorem \ref{thm:mainA} with $\mathbb{R}^n$ replaced by $V$, i.e., for the case $n=3$.
\medskip

Restricting to $V$, the following spherical $(r,\theta,\phi)$-coordinates are more convenient for calculation,
\begin{equation}\label{3d-spherical-coordinates}
x_1=r\cos\theta,\quad x_{k+1}=r\sin\theta\cos\phi,\quad x_n=r\sin\theta\sin\phi,
\end{equation}
with
$(r,\theta,\phi)\in\mathbb{R}_{>0}\times(0,\pi)\times (\mathbb{R}/(2\mathbb{Z}\pi))$.
Similarly, we use $t$ to denote the $\theta$-coordinate where ${F_1}_{|V}$ or
${g_1}_{|V}$ is concerned.
It is easy to check that $\Phi_{|C(U'_1)}$
fixes all $\phi$-coordinates.

When $k=1$, the $(r,\theta,\phi)$-coordinates are related to
the $(r,\theta,\xi)$-coordinates in Section \ref{subsection-3-1} by
$$(r,\theta,\phi)\leftrightarrow(r,\theta,\xi)=(r,\theta,(\cos\phi,0,\cdots,0,\sin\phi)^T).$$
 The functions $f(t)$, $h(\theta)$ and $\theta(t)$ in the
 $(r,\theta,\xi)$-coordinates presentation
 are completely inherited by the
 $(r,\theta,\phi)$-coordinates presentation when restricted to $V$, i.e.,
$${E_1}_{|V}=r^2 f(t),\quad {E_2}_{|V}=r^2 h(\theta), \quad \Phi_{|C(U'_1)}:(r,t,\phi)\mapsto(\tfrac{f(t)^{1/2}r}{h({\theta(t)})^{1/2}},
\theta(t),\phi).$$

When $k>1$, we can still use the spherical coordinates $(r,\theta,\phi)$ in (\ref{3d-spherical-coordinates}) on $V$, which
is related to the $(r,\theta,\xi',\xi'')$-coordinates by
\begin{eqnarray*}
& &(r,\theta,\phi)\leftrightarrow
(r,\theta,(1,0,\cdots,0)^T,(\cos\phi,0,\cdots,0,\sin\phi)^T),\quad
\forall\theta\in(0,\pi/2), \\
& &(r,\theta,\phi)\leftrightarrow
(r,\pi-\theta,(-1,0,\cdots,0)^T,(\cos\phi,0,\cdots,0,\sin\phi)^T),\quad
\forall\theta\in(\pi/2,\pi).
\end{eqnarray*}
Since in this case $C(U'_1)\subset V_{x_1>0}$ has positive $x_1$-coordinates, i.e.,
its $\theta$-coordinates range in $(0,\pi/2)$,
the functions
$f(t)$, $h(\theta)$ and $\theta(t)$  in the  $(r,\theta,\xi',\xi'')$-coordinates presentations, which are originally defined on $(0,\pi/2)$,
can still be applied to the discussion for
${\Phi}_{|C(U'_1)}$.
So to summarize, we have\medskip

\noindent{\bf Observation 2:}\
No matter $k=1$ or $k>1$, the functions $f(t)$, $h(\theta)$ and $\theta(t)$ in the spherical coordinates presentations
for the Minkowski norms $F_i$ and the local Hessian isometry $\Phi$
on $C(U_1)$
can be used to discuss the restriction $\Phi_{|C(U'_1)}$.
\medskip

We see from the next subsection, that the key steps in the proof, i.e., using the spherical $(r,\theta,\phi)$-coordinates to deduce and
analyse the ODE system for $\theta(t)$ and $h(\theta)$, and calculating the fundamental tensor for the linear and Legendre examples, are only relevant to the
$x_1$-, $x_{k+1}$- and $x_n$-coordinates. So they are all contained
in the proof of Theorem \ref{thm:mainA} when $n=3$. The functions $\theta(t)$ in (\ref{theta-linear}) and (\ref{theta-legendre}) for the linear example and Legendre example respectively are irrelevant to the dimension. So to summarize, we have\medskip

\noindent{\bf Conclusion:}
With some minor changes, the argument in the next subsection proves Theorem \ref{thm:mainA} generally.
\subsection{Proof of Theorem \ref{thm:mainA} when $n=3$}

Let $F_1$, $F_2$ be two Minkowski norms
on $\mathbb{R}^3$ which are invariant with respect to (the  same) standard  block diagonal action of $O(2)$ generated by the matrices of the form $\mathrm{diag}(1,A)$ with $A\in O(2)$. Their Hessian metrics are denoted as $g_1=g_1(\cdot,\cdot)$ and $g_2=g_2(\cdot,\cdot)$ respectively.

We fix the orthonormal coordinates $(x_1,x_2,x_3)$ such that
the $SO(2)$-action fixes each point on the line $V'$ presented by $x_2=x_3=0$ and rotates the plane $V''$ presented by $x_1=0$. We further require the marking point $y\in\mathbb{R}^3\backslash\{0\}$ has coordinates $(y_1,y_2,y_3)$ with $y_1\geq 0$, $y_2\geq 0$ and $y_3=0$.

In this subsection, we will only use the spherical coordinates
 $(r,\theta,\phi)\in\mathbb{R}_{>0}\times(0,\pi)\times
(\mathbb{R}/2\mathbb{Z}\pi)$ determined by
$$x_1=r\cos\theta,\quad x_2=r\sin\theta\cos\phi,\quad
x_3=r\sin\theta\sin\phi.$$
Then the $SO(2)$-action fixes $r$ and $\theta$ and shifts $\phi$.
We use $t$ to denote the $\theta$-coordinate and still call it the $\theta$-coordinate where $F_1$ or $g_1$ is concerned.

By the homogeneity and $SO(2)$-invariancy, $E_i=\tfrac12F_i^2$ can be presented
as
\begin{equation*}
E_1=r^2 {f(t)}\quad\mbox{and}\quad
E_2=r^2 {h(\theta)}
\end{equation*}respectively,
in which $f(t)$ and $h(\theta)$ are some even positive smooth functions on $\mathbb{R}$ with the period $2\pi$.

We have previously observed $y\notin V_1$, so we have $y_1\geq 0$ and $y_2>0$ for $y=(y_1,y_2,0)$, i.e., the $\theta$-coordinate of $y$ is contained in $(0,\pi/2]$, and the $\phi$-coordinate of $y$ is $0\in\mathbb{R}/(2\mathbb{Z}\pi)$. Without loss of generality, we assume $y\in S_{F_1}$. So its spherical coordinates can be presented as $(r_0,t_0,\phi_0)=(f(t_0)^{-1/2},t_0,0)$.
By (\ref{017}) and Lemma \ref{lemma-002}, we have the following at $y$:
\begin{eqnarray}
& &g_1(\tfrac{\partial}{\partial x_1},\tfrac{\partial}{\partial x_3})
=g_1(\tfrac{\partial}{\partial x_2},\tfrac{\partial}{\partial x_3})=0,\label{000}\\
& &g_1(\tfrac{\partial}{\partial x_1},\tfrac{\partial}{\partial x_2})\neq0,\label{019}\mbox{ and}\\
& &-\cos t_0\sin t_0\tfrac{{\rm d}^2}{{\rm d}t^2}f(t_0)+(\cos^2 t_0-\sin^2 t_0)\tfrac{{\rm d}}{{\rm d}t}f(t_0)\neq0.\label{018}
\end{eqnarray}

Let $\Phi$ be the local Hessian isometry from $F_1$ to $F_2$, which is defined on some $SO(2)$-invariant conic neighborhood of $y$ and maps orbits to orbits. By Theorem \ref{thm-decomposition} and
Lemma \ref{spherical-presenting-hessian-isometry-1}, we only need
to consider the situation that $\Phi$ fixes the $\phi$-coordinates and we can present it
by spherical coordinates as
\begin{equation}\label{020}
(r,t,\phi)\mapsto
(\tfrac{f(t)^{1/2}}{h(\theta(t))^{1/2}}\cdot r,\theta(t),\phi).
\end{equation}

We will first discuss the situation that $t_0\neq\pi/2,t'$, where
$t'$ is the unique solution of $\sin t \tfrac{{\rm d}}{{\rm d}t}f(t)-2\cos tf(t)=0$ in $(0,\pi)$.

Let $y(t)$ be a normal geodesic on $(S_{F_1},g_1)$ passing $y$, parametrized by the $\theta$-coordinate.
Using spherical coordinates, $y(t)$ can be locally presented as
$((2f(t))^{-1/2},t,0)$ around $y$, with its tangent vector field $\tfrac{{\rm d}}{{\rm d}t}y(t)=\tfrac{\partial}{\partial t}-\tfrac{1}{(2f(t))^{3/2}}\tfrac{{\rm d}}{{\rm d}t}f(t)\tfrac{\partial}{\partial r}$. By (\ref{eq:g}),
\begin{equation}\label{eq:theta1}
g_1(\tfrac{{\rm d}}{{\rm d}t}y(t),\tfrac{{\rm d}}{{\rm d}t}y(t))
=\tfrac{1}{2f(t)}\tfrac{{\rm d}^2}{{\rm d}t^2}f(t)-\tfrac{1}{4f(t)^2}\left(\tfrac{{\rm d}}{{\rm d}t}f(t)\right)^2+1.
\end{equation}
The $\Phi$-image $\gamma$ of the curve $y(t)$ is a curve on $S_{F_2}$ with constant $\phi$-coordinate $0$. When $\gamma=\gamma(\theta)$ is parametrized by the $\theta$-coordinate,
we similarly have
\begin{equation}\label{eq:theta1-1}
g_2(\tfrac{{\rm d}}{{\rm d}\theta}\gamma(\theta),
\tfrac{{\rm d}}{{\rm d}\theta}\gamma(\theta))
=\tfrac{1}{2h(\theta)}
\tfrac{{\rm d}^2}{{\rm d}\theta^2}h(\theta)-
\tfrac{1}{4h(\theta)^2}
\left(\tfrac{{\rm d}}{{\rm d}\theta}h(\theta)\right)^2+1.
\end{equation}
Since
$\Phi_*(\tfrac{{\rm d}}{{\rm d}t}y(t))=f'(t)\tfrac{{\rm d}}{{\rm d}\theta}\gamma(\theta(t))$, and $\Phi$ is a local isometry around $y=y(t_0)$,
we have
\begin{eqnarray}
& &\tfrac{1}{2f(t)}\tfrac{{\rm d}^2}{{\rm d}t^2}f(t)-\tfrac{1}{4f(t)^2}\left(\tfrac{{\rm d}}{{\rm d}t}f(t)\right)^2+1\nonumber\\\label{eq:t1}&=&
\left(\tfrac{{\rm d}}{{\rm d}t}\theta(t)\right)^2\cdot
\left(\tfrac{1}{2h(\theta)}
\tfrac{{\rm d}^2}{{\rm d}\theta^2}h(\theta(t))-
\tfrac{1}{4h(\theta)^2}
\left(\tfrac{{\rm d}}{{\rm d}\theta}h(\theta(t))\right)^2+1\right).
\end{eqnarray}

On the other hand, the equivariancy of $\Phi$ implies that  $\Phi_*(\tfrac{\partial}{\partial \phi})=\tfrac{\partial}{\partial \phi}$,
so by the isometric property of $\Phi$ and (\ref{eq:g}), we get
\begin{equation}\label{eq:t2}
\sin^2 t+\tfrac{\cos t\sin t}{2f(t)}\tfrac{{\rm d}}{{\rm d}t}f(t)
=\sin^2\theta(t)+
\tfrac{\cos \theta(t)\sin \theta(t)}{2h(\theta(t))}
\tfrac{{\rm d}}{{\rm d}\theta}h(\theta(t)).
\end{equation}

We view (\ref{eq:t1}) and (\ref{eq:t2}) as an ODE system
for the functions $\theta(t)$ and $h(\theta)$. We first determine $\theta(t)$. Rewrite (\ref{eq:t2}) as
\begin{equation}\label{006}
\tfrac{1}{h(\theta(t))}\tfrac{{\rm d}}{{\rm d}\theta}h(\theta(t))
=\left({2\sin^2 t+\tfrac{\cos t\sin t}{f(t)}\tfrac{{\rm d}}{{\rm d}t}f(t)}\right){\csc\theta(t)\sec\theta(t)}-2\tan\theta(t),
\end{equation}
and differentiate (\ref{006}) with respect to  $t$, we get
\begin{eqnarray}\label{007}
& &\tfrac{{\rm d}}{{\rm d}t}\theta(t)\cdot
\left(\tfrac{1}{h(\theta(t))}
\tfrac{{\rm d}^2}{{\rm d}\theta^2}h(\theta(t))-
\tfrac{1}{h(\theta(t))^2}\left(
\tfrac{{\rm d}}{{\rm d}\theta}h(\theta(t))
\right)^2
\right)\nonumber\\
&=&\tfrac{{\rm d}}{{\rm d}t}\theta(t)\cdot
{\left(2\sin^2 t+\tfrac{\cos t\sin t}{f(t)}
\tfrac{{\rm d}}{{\rm d}t}f(t)\right)
\left(\sec^2\theta(t)-\csc^2\theta(t)\right)}
-2\tfrac{{\rm d}}{{\rm d}t}\theta(t)\cdot\sec^2\theta(t)\nonumber\\
& &+\left({4\cos t\sin t+\tfrac{\cos^2t-\sin^2t}{f(t)}\tfrac{{\rm d}}{{\rm d}t}f(t)
-\tfrac{\cos t\sin t}{f(t)^2}\left(\tfrac{{\rm d}}{{\rm d}t}f(t)\right)^2
+\tfrac{\cos t\sin t}{f(t)}\tfrac{{\rm d}^2}{{\rm d}t^2}f(t)
}\right)\nonumber\\
& &\cdot{\csc\theta(t)\sec\theta(t)}.
\end{eqnarray}
We plug (\ref{006}) and (\ref{007}) into the right side of (\ref{eq:t1}) to erase $h(\theta(t))$ and its derivatives, then we get a formal quadratic equation for $\tfrac{{\rm d}}{{\rm d}t}\theta(t)$,
\begin{equation}\label{eq:square}
A\left(\tfrac{{\rm d}}{{\rm d}t}\theta(t)\right)^2+B\left(
\tfrac{{\rm d}}{{\rm d}t}\theta(t)\right)+C=0,
\end{equation}
in which
\begin{eqnarray*}
A&=&\frac{\cos t\sin t\left(\cos t \tfrac{{\rm d}}{{\rm d}t}f(t)+2\sin t f(t)\right)\left(\sin t \tfrac{{\rm d}}{{\rm d}t}f(t)-2\cos t f(t)\right)}{2f(t)^2\cos^2\theta(t)\sin^2\theta(t)},\\
B&=&\frac{\tfrac{\cos t\sin t}{f(t)}\tfrac{{\rm d}^2}{{\rm d}t^2}f(t)-\tfrac{\cos t\sin t}{f(t)^2}\left(\tfrac{{\rm d}}{{\rm d}t}f(t)\right)^2+\tfrac{\cos^2t-\sin^2t}{f(t)}\tfrac{{\rm d}}{{\rm d}t}f(t)
+4\cos t\sin t
}{\cos\theta(t)\sin\theta(t)},\\
C&=&-\tfrac{1}{f(t)}\tfrac{{\rm d}^2}{{\rm d}t^2}f(t)
+\tfrac{1}{2f(t)^2}\left(\tfrac{{\rm d}}{{\rm d}t}f(t)\right)^2-2.
\end{eqnarray*}

By (\ref{eq:t2}), $\theta_0=\theta(t_0)\in(0,\pi)$ equals $\pi/2$ iff
$t_0=\pi/2$ or $t'$, which has been excluded. So the denominators
in above calculation do not vanish. Meanwhile, we see the coefficient $A$ in (\ref{eq:square}) does not vanish for each value of $t$ (when it is sufficiently close to $t_0$).

Direct calculation shows that for each $t$, the two solutions of
(\ref{eq:square}) are
\begin{equation}\label{two-solutions}
\frac{\cos\theta(t)\sin\theta(t)}{\cos t\sin t}\quad\mbox{and}\quad
\tfrac{\left(-2f(t)\tfrac{{\rm {\rm d}^2}}{{\rm d}t^2}f(t)+\left(\tfrac{{\rm d}}{{\rm d}t}f(t)\right)^2-4f(t)^2\right)\cos\theta(t)\sin\theta(t)}{
\left(\cos t\tfrac{{\rm d}}{{\rm d}t}f(t)+2\sin t f(t)\right)\left(\sin t\tfrac{{\rm d}}{{\rm d}t}f(t)-2\cos t f(t)\right)}.
\end{equation}
The discriminant of (\ref{eq:square}) is
\begin{equation}\label{discriminant}
B^2-4AC=\left(\frac{\cos t\sin t \tfrac{{\rm d}^2}{{\rm d}t^2}f(t)+(\sin^2 t-\cos^2t)\tfrac{{\rm d}^2}{{\rm d}t}f(t)}{\cos\theta(t)\sin\theta(t)}\right)^2.
\end{equation}
By the inequality (\ref{018}), the discriminant is strictly positive when $t=t_0$. By continuity, we have immediately the
following lemma.
\begin{lemma} \label{lemma-000} Assume $t_0\in (0,\pi)\backslash\{\pi/2,t'\}$ satisfies (\ref{018}),
then one of the following
two cases must happen:
\begin{enumerate}
\item For all $t$ sufficiently close to $t_0$, we have
\begin{equation}\label{008}
\tfrac{{\rm d}}{{\rm d}t}\theta(t)=\tfrac{\cos\theta(t)\sin\theta(t)}{\cos t\sin t};
\end{equation}
\item For all $t$ sufficiently close to $t_0$, we have
\begin{equation}\label{009}
\tfrac{{\rm d}}{{\rm d}t}\theta(t)=\tfrac{\left(-2f(t)\tfrac{{\rm {\rm d}^2}}{{\rm d}t^2}f(t)+\left(\tfrac{{\rm d}}{{\rm d}t}f(t)\right)^2-4f(t)^2\right)\cos\theta(t)\sin\theta(t)}{
\left(\cos t\tfrac{{\rm d}}{{\rm d}t}f(t)+2\sin t f(t)\right)\left(\sin t\tfrac{{\rm d}}{{\rm d}t}f(t)-2\cos t f(t)\right)}.
\end{equation}
\end{enumerate}
\end{lemma}

Now we are ready to prove the following description for $\Phi$.

\begin{lemma}\label{thm:mainA-local}
Keep all above assumptions and notations for the $SO(2)$-invariant Minkowski norms $F_i$, the marking point $y\in\mathbb{R}^3\backslash\{0\}$ satisfying (\ref{eq:notzero}),
the local Hessian isometry $\Phi$ from $F_1$ to $F_2$ which is
defined around $y$, maps orbits to orbits and fixes all $\phi$-coordinates. Then there exists
a sufficiently small $SO(2)$-invariant conic open neighborhood $C(U_1)$
of $y$, such that either $\Phi_{|C(U_1)}$ coincides with
the restriction of a linear example, or it coincides with that of a Legendre example.
\end{lemma}

\begin{proof}
We  first prove Lemma \ref{thm:mainA-local}
with the assumption that the $\theta$-coordinate $t_0$ of $y$ satisfies $t_0\in(0,\pi)\backslash\{ \pi/2,t'\}$.

In each case of Lemma \ref{lemma-000}, the local Hessian isometry
$\Phi$ can be determined around $y$ for any given pair of
$\theta_0=\theta(t_0)\neq\pi/2$ and $h_0=h(\theta_0)>0$.
For example, in the case (1), we can use the ODE (\ref{008}) and its initial value condition  $\theta(t_0)=\theta_0$ to uniquely determine the function $\theta(t)$, and then use the ODE (\ref{eq:t2}) and its initial value condition $h(\theta_0)=h_0$
to uniquely determine $h(\theta)$. Then $\Phi$ is determined by
(\ref{020}) around $y$. Meanwhile, we see the ODE (\ref{008}) coincides with (\ref{ODE-1}), i.e., it is satisfied by the linear examples in Example \ref{example-linear}. With the parameter pair $(a,b)$ suitably chosen, both
initial value conditions can be met. So in this case, $\Phi$ is a linear isometry in some $SO(2)$-invariant conic neighborhood of $y$. In the case (2), the ODE (\ref{009}) coincides with (\ref{ODE-2}), i.e., it is satisfied by the Legendre examples in Example
\ref{example-Legendre}. We can suitably choose the parameter pair $(a,b)$ to meet
both initial value conditions. So in this case,
$\Phi$ coincides with a Legendre example in some $SO(2)$-invariant conic neighborhood of $y$.

Let us now prove Lemma \ref{thm:mainA-local} when $t_0=\pi/2$ or $t'$.

By (\ref{005}), for $t\neq t_0$ sufficiently close to $t_0$, we have $t\neq \pi/2,t'$ and
$$(\cos^2t-\sin^2t)\tfrac{{\rm d}}{{\rm d}t}f(t)-\cos t\sin t
\tfrac{{\rm d}^2}{{\rm d}t^2}f(t)\neq0.$$
Previous arguments indicate $\Phi$ is either a linear example
 or a Legendre example, when restricted to each side $t<t_0$ and $t>t_0$ respectively. When the restrictions of $\Phi$ to both sides are of the same type, by the smoothness of $\Phi$,  the parameter pairs $(a,b)$ for both
sides must coincide. The proof ends immediately in this case.

Finally, we prove that it can not happen that the restrictions
of $\Phi$ to the two sides of $t_0$ have different types. Assume conversely that it happens. For example, when $t<t_0$ (or $t>t_0$)
$\Phi$ is the linear example with the parameter pair $(a_1,b_1)$, and when $t>t_0$ (or $t<t_0$ respectively) $\Phi$ is the Legendre example with
the parameter pair $(a_2,b_2)$. Besides $b_1>0$ and $b_2>0$, we also have $a_1^{-1}a_2>0$ because $a_1$ and $a_2$ have the same sign as $\tfrac{{\rm d}}{{\rm d}t}\theta(t_0)$.
Using the linear example to calculate the fundamental tensor
$(b_{ij})=(g_2(\tfrac{\partial}{\partial x_i},\tfrac{\partial}{\partial x_j}))$ at $\Phi(y)$, we get
\begin{equation}\label{010}
\left(
  \begin{array}{cc}
    b_{11} & b_{12} \\
    b_{21} & b_{22} \\
  \end{array}
\right)=
\left(
  \begin{array}{cc}
    a_1^{-1} & 0 \\
    0 & b_1^{-1} \\
  \end{array}
\right)\left(
         \begin{array}{cc}
           a_{11} & a_{12} \\
           a_{21} & a_{22} \\
         \end{array}
       \right)
       \left(
         \begin{array}{cc}
           a_1^{-1} & 0 \\
           0 & b_1^{-1} \\
         \end{array}
       \right),
\end{equation}
in which $(a_{ij})=(g_1(\tfrac{\partial}{\partial x_i},\tfrac{\partial}{\partial x_j}))$ is the fundamental tensor of $F_1$
at $y$. Using the Legendre example to calculate $(b_{ij})$ at $\Phi(y)$, we get
\begin{equation}\label{011}
\left(
  \begin{array}{cc}
    b_{11} & b_{12} \\
    b_{21} & b_{22} \\
  \end{array}
\right)=
\left(
  \begin{array}{cc}
    a_2^{-1} & 0 \\
    0 & b_2^{-1} \\
  \end{array}
\right)\left(
         \begin{array}{cc}
           a^{11} & a^{12} \\
           a^{21} & a^{22} \\
         \end{array}
       \right)
       \left(
         \begin{array}{cc}
           a_2^{-1} & 0 \\
           0 & b_2^{-1} \\
         \end{array}
       \right),
\end{equation}
where $(a^{ij})_{1\leq i,j\leq 3}$ is the inverse matrix of $(a_{ij})_{1\leq i,j\leq 3}$.
Notice that $(a_{ij})_{1\leq i,j\leq3}$ is blocked-diagonal by (\ref{000}), so
\begin{equation}\label{012}
\left(
            \begin{array}{cc}
              a^{11} & a^{12} \\
              a^{21} & a^{11} \\
            \end{array}
          \right)=
          \left(
            \begin{array}{cc}
              a_{11} & a_{12} \\
              a_{21} & a_{22} \\
            \end{array}
          \right)^{-1}
. \end{equation}

Summarizing (\ref{010}), (\ref{011}) and (\ref{012}), we get
\begin{eqnarray*}
& &\left(
            \begin{array}{cc}
              a_1^{-2}a_2^2a_{11}^2+a_1^{-1}a_2b_1^{-1}b_2 a_{12}a_{21} & a_1^{-2}a_2^2a_{11}a_{12}+
              a_1^{-1}a_2b_1^{-1}b_2a_{12}a_{22} \\
              a_1^{-1}a_2b_{1}^{-1}b_2a_{11}a_{21}+b_1^{-2}b_2^2
              a_{21}a_{22}&
              a_1^{-1}a_2b_1^{-1}b_2a_{12}a_{21}+b_1^2b_2^2 a_{22}^2 \\
            \end{array}
          \right)\\
&=&\left[\left(
        \begin{array}{cc}
          a_1^{-1}a_2 & 0 \\
          0 & b_1^{-1}b_2 \\
        \end{array}
      \right)
      \left(
        \begin{array}{cc}
          a_{11} & a_{12} \\
          a_{21} & a_{22} \\
        \end{array}
      \right)
\right]^2=\left(
            \begin{array}{cc}
              1 & 0 \\
              0 & 1 \\
            \end{array}
          \right),
\end{eqnarray*}
from which we see
\begin{eqnarray*}
a_1^{-2}a_2a_{11}a_{12}+a_1^{-1}a_2b_1^{-1}b_2a_{12}a_{22}
=a_1^{-1}a_2a_{12}(a_1^{-1}a_2a_{11}+b_1^{-1}b_2a_{22})=0.
\end{eqnarray*}
Since $a_1^{-1}a_2>0$, $b_1>0$, $b_2>0$, $a_{11}>0$ and $a_{22}>0$, we get $a_{12}=g_1(\tfrac{\partial}{\partial x_1},\tfrac{\partial}{\partial x_2})=0$ at $y$. This is a contradiction to (\ref{019}).
\end{proof}


\begin{proof}[Proof of Theorem \ref{thm:mainA} when  $n=3$]
Let $C(U_1)$ be any
$SO(2)$-invariant connected  conic open subset of $\mathbb{R}^3$ in which
(\ref{eq:notzero}) is always satisfied, and $\Phi$ a local Hessian isometry from $F_1$ to $F_2$ which is defined in $C(U_1)$ and maps orbits to orbits. Without loss of generality, we assume $\Phi$ fixes all $\phi$-coordinates. Since by Lemma \ref{lemma-002} $F_1$ is nowhere locally Euclidean in $C(U_1)$, its Legendre transformation is nowhere locally linear in $C(U_1)$ either. So when we glue the local descriptions for $\Phi$  everywhere  in $C(U_1)$, the two cases in Lemma \ref{thm:mainA-local} can not be glued together. By the connectedness of $C(U_1)$ and the smoothness of $\Phi$, either $\Phi$ is uniformly locally modelled by the the same linear example everywhere in $C(U_1)$, or it is
uniformly locally modelled by the same Legendre example everywhere in $C(U_1)$. In either case, Theorem \ref{thm:mainA} when $n=3$
 is proved.
\end{proof}
\label{subsection-3-4}

\subsection{Proof of  Theorem \ref{thm:mainB}.}
If  (\ref{eq:zero}) is fulfilled at every point of $C(U_1)$, then by Lemma \ref{lemma-002}, the following ODE is satisfied:
 $$-\cos t\sin t\tfrac{{\rm d}^2}{{\rm d}t^2}f(t)+(\cos^2 t-\sin^2 t)
\tfrac{{\rm d}}{{\rm d}t}f(t)=0.$$
Its  solution is
$f(t)=c_1+c_2\cos 2t$ and  the corresponding Minkowski norms are Euclidean which proves the first statement of Theorem \ref{thm:mainB}.

In order to prove the remaining statements, observe that  by (\ref{eq:zero}) and Lemma \ref{lemma-002},
the ODEs (\ref{008}) and (\ref{009}) in Lemma \ref{lemma-000} coincide for almost all relevant values of $t$, i.e., the ODE
${\cos t\sin t}\tfrac{{\rm d}}{{\rm d}t}\theta(t)={\cos\theta(t)\sin\theta(t)}$ is satisfied in $C(U_1)$. Then we can explicitly solve $\theta(t)$ from this ODE, then solve $h(\theta)$ from (\ref{eq:t2}), and see that the corresponding isometry is
linear as we claimed in  Theorem \ref{thm:mainB}.

\section{Proof of Corollary \ref{thm:main2}.} \label{sec:prooflandsberg}
Let $F$ be a Finsler metric on
 $M$ with $\dim M=n\geq 3$. Assume that for some $k$ with $1\leq k\leq n/2$   and for each tangent space $T_pM$, the Minkowski norm $F_{|T_{p}M} $ is $SO(k)\times SO(n-k)$-invariant and that the Landsberg curvature of $F$ vanishes.

We need to show that
for  every  smooth curve $c:[0,1]\to M$ the Berwald parallel  transport  $\tau_{1}:T_{c(0)}M\to T_{c(1)}M$
 is linear.   As recalled in Section Theorem \ref{thm:mainB}, for each $s\in[0,1]$, the Berwald parallel transport
$\tau_{s}:T_{c(0)}M\to T_{c(s)}M$ along $c_{|[0,s]}$
is a Hessian isometry from $F_{|T_{c(0)}M}$ to  $F_{|T_{c(s)}M}$.

At each  tangent space $T_pM$   we consider the Hessian metric of $F_{|T_pM}$. If at the point $c(0)$
the  connected  isometry group $G_0$ of the Hessian metric is  bigger  than $SO(k)\times SO(n-k)$, then this is so at every point $p\in M$ (assumed connected) and by Theorem \ref{thm:main1} the metric $F$ is Riemannian and therefore Berwald.

If the connected  isometry group $G_0$ of the Hessian metric coincides with    $SO(k)\times SO(n-k)$, then every isometry $\tau_s$ maps orbits to orbits  so we can apply Theorems \ref{thm:mainA} and \ref{thm:mainB}.  Note that since  $\tau_1$ is
 positive  homogeneous, the condition that  $\tau_1$ is linear is equivalent to the condition that  the second partial derivativesof $\tau_1$
with respect to the linear variables in $T_{c(0)}M$ vanish. If this condition is fulfilled at almost every point of $T_{c(0)}M\setminus \{0\},$ it is fulfilled at every point.

Let us consider the conic open sets $C(U')$ and $C(U'')$ of $T_{c(0)}M\setminus \{0\}$
 as in Section \ref{section_results}: the set $C(U')$ contains all  $y$ such that \eqref{eq:notzero} is  fulfilled, and the set $C(U'')$ is the set of inner points of the compliment  $T_{c(0)}M\setminus (\{0\}\cup C(U'))$.  The union $C(U') \cup C(U'')$ is  dense in  $T_{c(0)}M$.

By Theorem \ref{thm:mainB} the restriction of $\tau_1$ to each connected component of $C(U'')$ is linear.

Let us show that  the restriction of $\tau_1$ to each connected component of $C(U')$  which we call $C(U'_1) $
is also  linear.
In order to do it, we consider the Legendre transformation $\Psi:T_{c(0)}M\to T_{c(0)}M$ corresponding to $F_{|T_{c(0)}M}$ and
the following   two subsets of the interval $[0,1]$:
 \begin{eqnarray*}
 T_1:&=& \{s\in  [0,1] \mid {\tau_s}_{|C(U'_1)}   \textrm{  is a linear transformation}\}\quad\mbox{and}\\
 T_2:&=&\{s\in [0,1]\mid {\tau_s}_{|C(U'_1)}\textrm{ is the composition of a linear transformation and }\Psi\}.
 \end{eqnarray*}
The subsets are disjunkt since $\Psi$ is not  Euclidean in $C(U_1)$  (see also Lemma \ref{lemma-002}).
They  satisfy  $T_1\cup T_2=[0,1]$ by Theorem \ref{thm:mainA}.
Notice that $\tau_s$ for $s\in[0,1]$ are a smooth family of Hessian isometries. $T_1$  can be  defined by the  condition that the second partial derivatives of $\tau_s$  vanish
for all $y\in C(U_1)$  and this is a finite system of equations. Similarly, $T_2$ can be   defined by the condition that the second partial derivatives of $\tau_s\circ \Psi $ vanish
for all $y\in C(U_1)$. So both $T_1$ and $T_2$ are closed subsets of $[0,1]$.
 By
the connectedness of $[0,1]$, one  of the sets $T_1$, $T_2$ must  be empty. But $T_1\ne \varnothing$, since $\tau_0$ is linear. Thus,    $T_1=[0,1]$ which implies that  ${\tau_1}_{|C(U'_1)}$ is  linear.

Finally, we have proved that  the restriction of $\tau_1$ to every connected component of an open everywhere dense subset of $T_{c(0)}M$  is linear; as explained above it implies that  $\tau_1$ is linear. Corollary  \ref{thm:main2} is proved.

\subsection*{ Acknowledgements.} The first author sincerely thanks Yantai University, Sichuan University, and Jena University
for hospitality during the preparation of this paper.
The first author is supported by National Natural Science
Foundation of China (No.~11821101, No.~11771331),
Beijing Natural Science Foundation
(No.~00719210010001, No.~1182006), Research Cooperation Contract (No.~Z180004), and Capacity Building for Sci-Tech  Innovation -- Fundamental Scientific Research Funds (No.~KM201910028021). The second author thanks  Capital Normal University for the  hospitality,
 Thomas Wannerer  for useful discussions and
DFG  for partial support via projects MA 2565/4 and  MA 2565/6.

\vspace{5mm}

\begin{thebibliography}{99}
  \bibitem{AlvarezPaiva2006} J. Alvarez Paiva, {  Some problems on Finsler geometry. }  Handbook of differential geometry. Vol. II, 1--33, Elsevier/North-Holland, Amsterdam, 2006.

\bibitem{An1995} G.S. Asanov, Finsler cases of GF-space, Aequationes Math. {\bf 49} (3) (1995), 234-251.

\bibitem{An1998} G.S. Asanov, Finslerian metric functions over the product $\mathbb{R}\times M$ and their potential applications, Rep. Math. Phys. {\bf 41} (1) (1998), 117-132.

\bibitem{As2006} G.S. Asanov, Finsleroid-Finsler space with Berwald and Landsberg conditions, Rep. Math. Phys. {\bf 58} (2006), 275-300.

\bibitem{As2007} G.S. Asanov, Finsleroid-Finsler space with geodesic spray coefficients, Publ. Math. Debrecen {\bf 71} (2007), 397-412.





\bibitem{Bao2007} D.  Bao, {  On two curvature-driven problems in Riemann-Finsler geometry.} Finsler geometry, Sapporo 2005 - in memory of Makoto Matsumoto, 19--71, Adv. Stud. Pure Math., {\bf 48}(2007), Math. Soc. Japan, Tokyo.

 \bibitem{BaoChernShen1997} D. Bao, S.S.  Chern, Z.  Shen,  {   Rigidity issues on Finsler surfaces. }  Rev. Roumaine Math. Pures Appl. {\bf 42}(1997) 707--735.

	\bibitem{BCS} D. Bao, S.S.  Chern, Z.  Shen,
An introduction to Riemann-Finsler geometry,
Graduate Texts in Mathematics, 200. Springer-Verlag, New York, 2000.
	
\bibitem{Berwald1941} L.   Berwald, {  Ueber Finslersche und Cartansche Geometrie. I. Geometrische Erkl\"arungen der Kr\"ummung und des Hauptskalars eines zweidimensionalen Finslerschen Raumes,} (German) Mathematica, Timisoara {\bf 17}(1941),  34--58.
	
 \bibitem{Berwald1947} L.   Berwald, {  Ueber Finslersche und Cartansche Geometrie. IV. Projektivkr\"ummung allgemeiner affiner R\"aume und Finslersche R\"aume skalarer Kr\"ummung, }  Ann. of Math. (2) {\bf 48 }(1947),  755--781.

\bibitem{Bl1906} G. A.  Bliss,  { A generalization of the notion of angle,} Trans. Amer. Math. Soc. {\bf 7}(1906), no. 2, 184--196.

\bibitem{BMR}  A.V. Bolsinov, V. S. Matveev, S. Rosemann,  {
Local normal forms for c-projectively equivalent metrics and proof of the Yano-Obata conjecture in arbitrary signature. Proof of the projective Lichnerowicz conjecture for Lorentzian metrics},  arXiv:1510.00275  .

\bibitem{Br1967} F. Brickell, A theorem on homogeneous functions,    J. London Math. Soc, 42 (1967), 325--329.

\bibitem{Cartan} \'El.   Cartan, Sur les espaces de Finsler, C. R. Acad. Sci. Paris, 196(1933), 582--586.

\bibitem{Yau} S. Y. Cheng, S.-T. Yau,
The real Monge-Ampere equation and affine flat structures, Proceedings of the 1980 Beijing Symposium on Differential Geometry and Differential Equations, Vol. 1, 2, 3 (Beijing, 1980), 339--370, Sci. Press Beijing, Beijing, 1982.

\bibitem{CS2005} S.-S. Chern and Z. Shen, Riemann-Finsler Geometry, Nankai Tracts in Mathematics {\bf 6},
    World Scientific, 2005.

\bibitem{Cr2011}  M. Crampin, { On Landsberg spaces and the Landsberg-Berwald problem,} Houston J. Math. {\bf 37}(2011), no. 4, 1103--1124.

\bibitem{Cr2020} M. Crampin, {  Finsler spaces of $(\alpha, \beta)$ type and semi-C-reducibility}, preprint, (2020).


\bibitem{DX2015} S. Deng and M. Xu,
Left invariant Clifford-Wolf homogeneous
$(\alpha,\beta)$-metrics on compact
semisimple Lie groups, Transform. Groups
{\bf20} (2) (2015), 395--416.




\bibitem{DX2016} S. Deng and M. Xu, $(\alpha_1,\alpha_2)$-Metrics
and Clifford-Wolf homogeneity, J. Geom. Anal. {\bf26} (2016), 2282--2321.




\bibitem{Dodson2006} C. Dodson, { A short review on Landsberg spaces,}  Workshop on Finsler and semi-Riemannian geometry, 24--26 May 2006, San Luis Potosi, Mexico.


\bibitem{FL2018} H. Feng, M. Li, An equivalence theorem for
a class of Minkowski spaces, preprint, (2018), arXiv:1812.11938v1.

\bibitem{GD} I. M. Gelfand, I. Ja. Dorfman, Hamiltonian operators and algebraic structures
 associated with them, (Russian) Funktsional. Anal. i Prilozhen. {\bf13}(1979), no. 4,
13--30, 96.

\bibitem{Ha1903} G. Hamel, {\"Uber die Geometrieen, in denen die Geraden die K\"urzesten sind,} (German) Math. Ann. {\bf 57}(1903), no. 2, 231--264.



\bibitem{Is1955}  S. Ishihara, {
Homogeneous Riemannian spaces of four dimensions,}
J. Math. Soc. Japan {\bf 7}(1955), 345--370.

\bibitem{MZ2002} M. Ji, Z.  Shen,
{ On strongly convex indicatrices in Minkowski geometry,}
Canad. Math. Bull. 45 (2002), no. 2, 232--246.

\bibitem{Kozma}  L. Kozma, On holonomy groups of Landsberg manifolds. Tensor (N.S.) {\bf 62}(2000), 87--90.

\bibitem{Landsberg1908} G. Landsberg, {\it \"Uber die Kr\"ummung in der Variationsrechnung,}  Math. Ann. {\bf 65}(1908),  313--349.


\bibitem{La1967} D. Laugwitz, Differentialgeometrie in Vektorr\"aumen, unter besonderer Ber\"ucksichtigung der unendlichdimensionalen R\"aume. Braunschweig 1965.

\bibitem{simon} A. M. Li, U. Simon,   G. S. Zhao,  Global affine differential geometry of hypersurfaces. De Gruyter Expositions in Mathematics, {\bf  11}. Walter de Gruyter \& Co., Berlin, 1993.

\bibitem{Ma1974} M. Matsumoto, On Finsler spaces with Randers metric and special forms of important tensors. J. Math. Kyoto Univ. {\bf 14} (1974), 477--498.


\bibitem{Ma1996} M. Matsumoto,  { Remarks on Berwald and Landsberg spaces,}  Finsler geometry (Seattle, WA, 1995), 79--82,
Contemp. Math. {\bf  196}, Amer. Math. Soc., Providence, RI, 1996.

\bibitem{Matveev2009b} V. S. Matveev, {On ``All regular Landsberg metrics are always Berwald" by Z. I. Szabo,}  Balkan Journ. Geom. {\bf  14}(2009),  50--52.



\bibitem{MS1979} M. Matsumoto, C. Shibata, On semi-C-reducibility, T-tensor$=0$, and S4-likeness of
    Finsler spaces, J. Math. Kyoto Univ. {\bf 19} (2) (1979),
    301-314.

\bibitem{MT2012}
V.~S.~Matveev, M.~Troyanov,
 The Binet-Legendre metric in Finsler geometry,
Geom. Topol. {\bf  16} (2012), 2135--2170.





\bibitem{MZ2014} X. Mo and L. Zhou, The curvatures of spherically
symmetric Finsler metrics in $\mathbb{R}^n$, arXiv:1202.4543.


\bibitem{Ob1955} M. Obata, On $n-$dimensional homogeneous spaces of Lie groups of dimension greater than $n(n-1)/2$,
 J. Math. Soc. Japan {\bf 7}(1955), 371--388.


 \bibitem{Rund} H.
Rund, { The differential geometry of Finsler spaces, } Die Grundlehren der Mathematischen Wissenschaften. {\bf 101}(1959). Berlin-G\"ottingen-Heidelberg: Springer-Verlag.


\bibitem{Sc1968} R. Schneider,
{ \"Uber die Finslerr\"aume mit $S_{ijkl}=0$},  (German)
Arch. Math. (Basel) {\bf 19}(1968), 656--658.

\bibitem{Sc2013} R. Schneider, Convex Bodies: The Brunn-Minkowski Theory, 2nd ed., Cambridge University Press, 2013.

\bibitem{Sh2007} H. Shima, The Geometry of Hessian Structures, World Scientific, 2007.

\bibitem{Sh2013}  H. Shima,  Geometry of Hessian structures. Geometric science of information, 37--55, Lecture Notes in Comput. Sci., 8085, Springer, Heidelberg, 2013.

\bibitem{Sh2001} Z. Shen, Lectures on Finsler geometry, World Scientific, 2001.


\bibitem{Shen2009a}  Z. Shen, { Some open problems in Finsler geometry}, \\ \url{https://www.math.iupui.edu/~zshen/Research/papers/Problem.pdf} (Posted in   2009).
\bibitem{Sh2009} Z. Shen, On a class of Landsberg metrics in Finsler geometry, Canad. J. Math. {\bf 61} (2009), 1357-1374.


\bibitem{Sz1981} Z. I. Szab\'o,  {  Positive definite Berwald spaces (Structure theorems)},
Tensor N. S. {\bf 35}(1981) 25--39.


\bibitem{XD2019} M. Xu and S. Deng, The Landsberg equation of a Finsler space, Ann. Sc. Norm. Super. Pisa Cl. Sci. (2019), doi:10.2422/2036-2145.201809\_015, \url{arXiv:1404.3488}.





\bibitem{Ya1953} K. Yano, {
On $n$-dimensional Riemannian spaces admitting a group of motions of order $n(n-1)/2+1$,}
Trans. Amer. Math. Soc. {\bf 74}(1953), 260--279.

\bibitem{ZWL2019} S. Zhou, J. Wang and B. Li, On a class of almost regular Landsberg metrics, Sci. China Math. {\bf 62} (5) (2019), 935-960.

\end{thebibliography}
\end{document}